\documentclass[preprint, 11pt]{imsart}

\usepackage[square,numbers]{natbib}
\RequirePackage[colorlinks,citecolor=blue,urlcolor=blue]{hyperref}
\usepackage{mathrsfs,amsthm,amsmath}
\usepackage{amssymb,multirow}
\usepackage{amsfonts}
\usepackage{stmaryrd}
\usepackage{dsfont}

\usepackage{ucs}
\usepackage[utf8x]{inputenc}
\usepackage{amsmath}
\usepackage{amsfonts}
\usepackage{amssymb}
\usepackage{amsthm}
\usepackage{fontenc}
\usepackage{graphicx}

\usepackage{bbm}
\usepackage{bm}
\RequirePackage[OT1]{fontenc}

\startlocaldefs

\newcommand{\E}{\mathds{E}}

\renewcommand{\P}{\mathds{P}}

\def\simiid{\stackrel{\mbox{\scriptsize{iid}}}{\sim}}
\theoremstyle{plain}
\newtheorem{thm}{\textsc{Theorem}}

\newtheorem{lem}{\textsc{Lemma}}

\newtheorem{prp}{\textsc{Proposition}}

\allowdisplaybreaks
\endlocaldefs

\begin{document}

\begin{frontmatter}

\title{Large deviation principles for the Ewens-Pitman sampling model}
\runtitle{Large deviations for the Ewens-Pitman sampling model}

\author{\fnms{Stefano Favaro}\thanksref{t1}\ead[label=e1]{stefano.favaro@unito.it}}
\and
\author{\fnms{Shui Feng}\ead[label=e2]{shuifeng@univmail.cis.mcmaster.ca}}

\address{
Department of Economics and Statistics\\ University of Torino\\ Corso Unione Sovietica 218/bis, I--10134 Torino, Italy \\
 \printead{e1}}

\address{Department of Mathematics and Statistics\\ McMaster University\\ Hamilton, Canada  L8S 4K1 \\ \printead{e2}}

\runauthor{S. Favaro and S. Feng}

\thankstext{t1}{Also affiliated to Collegio Carlo Alberto, Moncalieri, Italy.}

\affiliation{University of Torino and McMaster University}

\vspace{0.4cm}
\begin{abstract}
Let $M_{l,n}$ be the number of blocks with frequency $l$ in the exchangeable random partition induced by a sample of size $n$ from the Ewens-Pitman sampling model. We show that, as $n$ tends to infinity, $n^{-1}M_{l,n}$ satisfies a large deviation principle and we characterize the corresponding rate function. A conditional counterpart of this large deviation principle is also presented. Specifically, given an initial sample of size $n$ from the Ewens-Pitman sampling model, we consider an additional sample of size $m$. For any fixed $n$ and as $m$ tends to infinity, we establish a large deviation principle for the conditional number of blocks with frequency $l$ in the enlarged sample, given the initial sample. Interestingly, the conditional and unconditional large deviation principles coincide, namely there is no long lasting impact of the given initial sample. Potential applications of our results are discussed in the context of Bayesian nonparametric inference for discovery probabilities.
\end{abstract}



\end{frontmatter}

\section{Introduction}

The present paper focuses on exchangeable random partitions induced by the so-called Ewens-Pitman sampling model, first introduced in Pitman (1995). Let $\mathbb{X}$ be a Polish space and let $\nu$ be a nonatomic probability measure on $\mathbb{X}$. For any $\alpha\in[0,1)$ and $\theta>-\alpha$ let us consider a sequence  $(X_{i})_{i\geq1}$ of $\mathbb{X}$-valued random variables such that $\P[X_{1}\in\cdot]=\nu(\cdot)$, and for any $i\geq1$
\begin{equation}\label{predict}
\P[X_{i+1}\in\cdot\,|\,X_{1},\ldots,X_{i}]=\frac{\theta+j\alpha}{\theta+i}\nu(\cdot)+\frac{1}{\theta+i}\sum_{l=1}^{j}(n_{l}-\alpha)\delta_{X_{l}^{\ast}}(\cdot)
\end{equation}
with $X_{1}^{\ast},\ldots,X_{j}^{\ast}$ being the $j$ distinct values in $(X_{1},\ldots,X_{i})$ with frequencies $\mathbf{n}=(n_{1},\ldots,n_{j})$. The predictive distribution \eqref{predict} is referred to as the Ewens-Pitman sampling model. Pitman \cite{Pit(95)} showed that the sequence $(X_{i})_{i\geq1}$ generated by \eqref{predict} is exchangeable and its de Finetti measure $\Pi$ is the distribution of the two parameter Poisson-Dirichlet process $\tilde{P}_{\alpha,\theta,\nu}$ introduced in Perman et al. \cite{Per(92)}, i.e.
\begin{align}\label{eq:bnpmodel}
X_i\,|\,\tilde P_{\alpha,\theta,\nu} & \quad\simiid\quad \tilde P_{\alpha,\theta,\nu}\qquad i=1,\ldots,n\\[4pt]
\notag\tilde P_{\alpha,\theta,\nu} & \quad\sim\quad \Pi,
\end{align}
for any $n\geq1$. See Pitman and Yor \cite{Pit(97)} for details on $\tilde{P}_{\alpha,\theta,\nu}$. For $\alpha=0$ the Ewens-Pitman sampling model reduces to the celebrated sampling model by Ewens \cite{Ewe(72)}. The Ewens-Pitman sampling model plays an important role in several research areas such as population genetics, Bayesian nonparametrics, machine learning, combinatorics and statistical physics. See Pitman \cite{Pit(06)} for a detailed account.

According to \eqref{predict} and \eqref{eq:bnpmodel}, a sample $(X_{1},\ldots,X_{n})$ from $\tilde{P}_{\alpha,\theta,\nu}$ induces an exchangeable random partition of $\{1,\ldots,n\}$ into $K_{n}$ blocks with frequencies $\mathbf{N}_{n}=(N_{1},\ldots,N_{K_{n}})$. See Pitman \cite{Pit(95)} for details. Such a random partition has been the subject of a rich and active literature and, in particular, there have been several studies on the large $n$ asymptotic behaviour of $K_{n}$. Specifically, for any $\alpha\in(0,1)$ and $q > -1$, let $S_{\alpha,q\alpha}$ be a positive random variable such that
\begin{equation}\label{eq:sdiversity}
\P[S_{\alpha, q\alpha}\in dy]=\frac{\Gamma(q\alpha+1)}{\alpha\Gamma(q+1)}y^{q-1-1/\alpha}f_{\alpha}(y^{-1/\alpha})dy,
\end{equation}
where $f_{\alpha}$ denotes the density function of a positive $\alpha$-stable random variable. For any $\alpha\in(0,1)$ and $\theta>-\alpha$ Pitman \cite{Pit(96)} established a fluctuation limit for $K_{n}$, namely
\begin{equation}\label{eq:asim_prior_2pd}
\lim_{n\rightarrow+\infty}\frac{K_{n}}{n^{\alpha}}=S_{\alpha,\theta}\qquad\text{a.s.}
\end{equation}
Furthermore, let $M_{l,n}$ be the number of blocks with frequency $l\geq1$ such that $K_{n}=\sum_{1\leq l\leq n}M_{l,n}$ and $n=\sum_{1\leq l\leq n}lM_{l,n}$. Then, Pitman \cite{Pit(06)} showed that
\begin{equation}\label{eq:asim_prior_2pd_freq}
\lim_{n\rightarrow+\infty}\frac{M_{l,n}}{n^{\alpha}}=\frac{\alpha(1-\alpha)_{(l-1)}}{l!}S_{\alpha,\theta}\qquad\text{a.s.}
\end{equation}
where $(x)_{(n)}=(x)(x+1)\cdots(x+n-1)$ denotes the rising factorial of $x$ of order $n$ with the proviso $(x)_{(0)}=1$. In contrast, for $\alpha=0$ and $\theta>0$, $K_{n}$ and $M_{l,n}$ have a different asymptotic behaviour. Specifically, Korwar and Hollander \cite{Kor(73)} and Arratia et al. \cite{Arr(92)} showed that $\lim_{n\rightarrow+\infty}K_{n}/\log n=\theta$ and $\lim_{n\rightarrow+\infty}M_{l,n}=P_{\theta/l}$ almost surely, where $P_{\theta/l}$ is distributed according to a Poisson distribution with parameter $\theta/l$. See Arratia et al. \cite{Arr(03)}, Barbour and Gnedin \cite{Bar(09)} and Schweinsberg \cite{Sch(10)} for recent generalizations and refinements of  \eqref{eq:asim_prior_2pd} and \eqref{eq:asim_prior_2pd_freq}.

Feng and Hoppe \cite{Fen(98)} further investigated the large $n$ asymptotic behaviour of the random variable $K_{n}$ and, in particular, they established a large deviation principle for $K_{n}$. Specifically, for any $\alpha\in(0,1)$ and $\theta>-\alpha$, they showed that $n^{-1}K_{n}$ satisfies a large deviation principle with speed $n$ and rate function
\begin{equation}\label{rate_prior_2pd_lamb}
I^{\alpha}(x)=\sup_{\lambda}\{\lambda x-\Lambda_{\alpha}(\lambda)\}
\end{equation}
where $\Lambda_{\alpha}(\lambda)=-\log(1-(1-e^{-\lambda})^{1/\alpha})\mathbbm{1}_{(0,+\infty)}(\lambda)$. In contrast, for $\alpha=0$ and $\theta>0$, it was shown by Feng and Hoppe \cite{Fen(98)} that  $(\log n)^{-1}K_{n}$ satisfies a large deviation principle with speed $\log n$ and rate function of the following form
 \begin{displaymath}
I_{\theta}(x)=\left\{\begin{array}{ll}
x\log \frac{x}{\theta}-x+\theta&\qquad\text{ }x>0\\[4pt]
\theta&\qquad\text{  }x=0\\[4pt]
+\infty&\qquad\text{  }x<0.
\end{array}\right.
\end{displaymath}
It is worth pointing out that the rate function \eqref{rate_prior_2pd_lamb} depends only on the parameter $\alpha$, which displays the different roles of the two parameters, $\alpha$ and $\theta$, at different scales. We refer to Feng and Hoppe \cite{Fen(98)} for an intuitive explanation in terms of a Poisson embedding scheme for the Ewens-Pitman sampling model.

In this paper we establish a large deviation principle for $M_{l,n}$. Specifically for any $\alpha\in(0,1)$ and $\theta>-\alpha$ we show that, as $n$ tends to infinity, $n^{-1}M_{l,n}$ satisfies a large deviation principle with speed $n$ and we characterize the corresponding rate function $I^{\alpha}_{l}$. We also present a conditional counterpart of this large deviation principle. To this end, with a slightly abuse of notation, we write $X\,|\,Y$ to denote a random variable whose distribution coincides with the conditional distribution of $X$ given $Y$. Moreover, let $(X_{1},\ldots,X_{n})$ be an initial sample from $\tilde{P}_{\alpha,\theta,\nu}$ featuring $K_{n}=j$ blocks with corresponding frequencies $\mathbf{N}_{n}=\mathbf{n}$, and let $(X_{n+1},\ldots,X_{n+m})$ be an additional unobserved sample. Recently, Lijoi et al. \cite{Lij(07)}, Favaro et al. \cite{Fav(09)} and Favaro et al. \cite{Fav(13)} derived and investigated the conditional distributions of the number $K_{m}^{(n)}$ of new blocks in $(X_{n+1},\ldots,X_{n+m})$ and of the number $M_{l,m}^{(n)}$ of blocks with frequency $l\geq1$ in $(X_{1},\ldots,X_{n+m})$, given $(X_{1},\ldots,X_{n})$. In particular, they showed that
\begin{equation}\label{eq:fluct_post}
\lim_{m\rightarrow+\infty}\frac{K_{m}^{(n)}}{m^{\alpha}}\,|\,(K_{n}=j,\mathbf{N}_{n}=\mathbf{n})=S_{\alpha,\theta}^{(n,j)}\qquad\text{a.s.}
\end{equation}
and 
\begin{equation}\label{eq:fluct_post_freq}
\lim_{m\rightarrow+\infty}\frac{M_{l,m}^{(n)}}{m^{\alpha}}\,|\,(K_{n}=j,\mathbf{N}_{n}=\mathbf{n})=\frac{\alpha(1-\alpha)_{(l-1)}}{l!}S_{\alpha,\theta}^{(n,j)}\qquad\text{a.s.}
\end{equation}
where $S^{(n,j)}_{\alpha,\theta}\stackrel{\text{d}}{=}B_{j+\theta/\alpha,n/\alpha-j}S_{\alpha,\theta+n}$ with $B_{j+\theta/\alpha,n/\alpha-j}$ and $S_{\alpha,\theta+n}$ being independent and distributed according to a Beta distribution with parameter $(j+\theta/\alpha,n/\alpha-j)$ and according to \eqref{eq:sdiversity} with $q=\theta+n$, respectively. Intuitively, as suggested by the fluctuations \eqref{eq:asim_prior_2pd_freq} and \eqref{eq:fluct_post_freq}, one may expect that $M_{l,n}$ and $M_{l,m}^{(n)}\,|\,(K_{n},\mathbf{N}_{n})$ have different asymptotic behaviours also in terms of large deviations, as $n$ and $m$ tend to infinity, respectively. Here we show that, for any fixed $n$ and as $m$ tends to infinity, $m^{-1}M_{l,m}^{(n)}\,|\,(K_{n},\mathbf{N}_{n})$ satisfies a large deviation principle with speed $m$ and rate function $I_{l}^{\alpha}$. In other terms, we show that there is no long lasting impact of the given initial sample to the large deviations. A similar behaviour was recently observed in Favaro and Feng \cite{Fav(14)} with respect to the large deviation principles for $K_{n}$ and $K_{m}^{(n)}\,|\,(K_{n},\mathbf{N}_{n})$.

The problem of studying conditional properties of exchangeable random partitions was first considered in Lijoi et al. \cite{Lij(08)}. Such a problem consists in evaluating, conditionally on the random partition $(K_{n},\mathbf{N}_{n})$ induced by a sample $(X_{1},\ldots,X_{n})$ from $\tilde{P}_{\alpha,\theta,\nu}$, the distribution of statistics from an additional sample $(X_{n+1},\ldots,X_{n+m})$. As observed in Lijoi et al. \cite{Lij(07)}, these statistics have direct applications in Bayesian nonparametric inference for species sampling problems arising from ecology, biology, genetic, linguistic, etc. Indeed, from a Bayesian perspective, \eqref{eq:bnpmodel} is a nonparametric model for the individuals $X_{i}$'s from a population with infinite species, where $\Pi$ is the prior distribution on the composition of such a population. The aforementioned $M_{l,m}^{(n)}$ is a representative statistic of practical interest. See, e.g., Griffiths and Span\`o \cite{Gri(07)} and  Bacallado et al. \cite{Bac(13)} for other statistics. In particular $\P[M_{l,m}^{(n)}=m_{l}\,|\,K_{n}=j,\mathbf{N}_{n}=\mathbf{n}]$ takes on the interpretation of the posterior distribution of the number of species with frequency $l$ in the enlarged sample $(X_{1},\ldots,X_{n+m})$, given $(X_{1},\ldots,X_{n})$ features $j$ species with frequencies $\mathbf{n}$. Hence $\E_{\alpha,\theta}[M_{l,m}^{(n)}\,|\,K_{n}=j,\mathbf{N}_{n}=\mathbf{n}]$ is the corresponding Bayesian nonparametric estimator under a squared loss function. In such a framework our conditional large deviation principle provides a large $m$ approximation of the estimator $\P[m^{-1}M_{l,m}^{(n)}\geq x\,|\,K_{n}=j,\mathbf{N}_{n}=\mathbf{n}]$, for any $x\geq0$. For large $m$ this is the right tail of the posterior proportion of species with frequency $l$ in the enlarged sample.

A closer inspection of the fluctuations \eqref{eq:fluct_post} and \eqref{eq:fluct_post_freq} reveals that for $l=1$ our conditional large deviation principle has a natural interpretation in the context of Bayesian nonparametric inference for discovery probabilities. Indeed, let $\P[D_{m}^{(n)}\in\cdot\,|\,K_{n}=j,\mathbf{N}_{n}=\mathbf{n}] $ be the conditional, or posterior, distribution of the probability of discovering a new species at the $(n+m+1)$-th draw, given the random partition $(K_{n},\mathbf{N}_{n})$ induced by $(X_{1},\ldots,X_{n})$. The additional sample $(X_{n+1},\ldots,X_{n+m})$ is assumed to be not observed. For large $m$, we show that $\P[D_{m}^{(n)}\in\cdot\,|\,K_{n}=j,\mathbf{N}_{n}]$ and $\P[m^{-1}M_{1,m}^{(n)}\in \cdot\,|\,K_{n}=j,\mathbf{N}_{n}=\mathbf{n}]$ are approximately equal. Accordingly our conditional large deviation principle provides a large $m$ approximation of the Bayesian nonparametric estimator $\P[D_{m}^{(n)}\geq x\,|\,K_{n}=j,\mathbf{N}_{n}=\mathbf{n}]$. Similarly, $\E_{\alpha,\theta}[m^{-1}M_{1,m}^{(n)}\,|\,K_{n}=j,\mathbf{N}_{n}=\mathbf{n}]$ provides a large $m$ approximation of the estimator of the probability of discovering a new species at the $(n+m+1)$-th draw, namely $\E_{\alpha,\theta}[D_{m}^{(n)}\,|\,K_{n}=j,\mathbf{N}_{n}=\mathbf{n}]$, which first appeared in Lijoi et al. \cite{Lij(07)}. An illustration of these asymptotic estimators is presented by using a genomic dataset. The interest in $\E_{\alpha,\theta}[D_{m}^{(n)}\,|\,K_{n}=j,\mathbf{N}_{n}=\mathbf{n}]$ and $\P[D_{m}^{(n)}\geq x\,|\,K_{n}=j,\mathbf{N}_{n}=\mathbf{n}]$, as well as in their large $m$ approximations, is related to the problem of determining the optimal sample size in species sampling problems. Indeed this problem is typically faced by setting a threshold $\tau$ for an exact or approximate mean functional of  $\P[D_{m}^{(n)}\in\cdot\,|\,K_{n}=j,\mathbf{N}_{n}=\mathbf{n}]$, and then making inference on the sample size $m$ for which this mean functional falls below, or above, $\tau$. This introduces a criterion for evaluating the effectiveness of further sampling.

The paper is structured as follows. In Section 2 we present the main result of the paper, namely the large deviation principle for $M_{l,n}$. Section 3 contains the conditional counterpart of this large deviation principle. In Section 4 we discuss potential applications of our conditional large deviation principle in the context of Bayesian nonparametric inference for species sampling problems.


\section{Large deviations for $M_{l,n}$} For any $\alpha\in(0,1)$ and $\theta>-\alpha$ the large deviation principle for $M_{l,n}$ is established through a detailed study of the moment generating function of $M_{l,n}$. This is in line with the approach originally adopted in Feng and Hoppe \cite{Fen(98)} for $K_{n}$. For any $\lambda>0$ let $y=1-e^{-\lambda}$ and 
\begin{equation}\label{eq_genfun_prior}
G_{M_{l,n}}(y;\alpha,\theta)=\E_{\alpha,\theta}\left[\left(\frac{1}{1-y}\right)^{M_{l,n}}\right]=\sum_{i\geq0}\frac{y^{i}}{i!}\mathbb{E}_{\alpha,\theta}[(M_{l,n})_{(i)}]
\end{equation}
be the moment generating function of the random variable $M_{l,n}$. Let $(y)_{[n]}=y(y-1)\cdots(y-n+1)$ be the falling factorial of $y$ of order $n$, with the proviso $(y)_{[0]}=1$. Proposition 1 in Favaro et al. \cite{Fav(13)} provides an explicit expression for $\mathbb{E}_{\alpha,\theta}[(M_{l,n})_{[r]}]$. Recalling that $(y)_{(n)}=\sum_{0\leq i\leq n}\sum_{0\leq j\leq i}|s(n,i)|S(i,j)(y)_{[j]}$, where $s$ and $S$ denote the Stirling number of the first type and the second type, an explicit expression for $\mathbb{E}_{\alpha,\theta}[(M_{l,n})_{(r)}]$ is obtained. Specifically, we have
\begin{equation}\label{prior_rising_1}
\E_{\alpha,\theta}[(M_{l,n})_{(r)}]=r!\sum_{i=0}^{r}{r-1\choose r-i}\frac{\left(\alpha\frac{(1-\alpha)_{(l-1)}}{l!}\right)^{i}\left(\frac{\theta}{\alpha}\right)_{(i)}(n)_{[il]}(\theta+i\alpha)_{(n-il)}}{i!(\theta)_{(n)}}
\end{equation}
and
\begin{equation}\label{prior_rising_2}
\E_{\alpha,0}[(M_{l,n})_{(r)}]=(r-1)!\sum_{i=0}^{r}{r\choose i}\frac{ \left(\alpha\frac{(1-\alpha)_{(l-1)}}{l!}\right)^{i}(n)_{[il]}(i\alpha)_{(n-il)}}{\alpha\Gamma(n)}, 
\end{equation}
where the sums over $i$ is nonnull for $0\leq i\leq \min(r,\lfloor{n/l\rfloor})$. In the next lemma we provide an explicit expression for the moment generating function $G_{M_{l,n}}(y;\alpha,0)$. This result follows by combining \eqref{prior_rising_2} with the series expansion on the right-hand side of \eqref{eq_genfun_prior}, and by means of standard combinatorial manipulations.

\begin{lem}\label{lemma_prior}
For any $\alpha\in(0,1)$
\begin{align}\label{mgf_prior}
&G_{M_{l,n}}(y;\alpha,0)\\
&\notag\quad=\sum_{i=0}^{\lfloor{n/l\rfloor}}\left(\frac{y}{1-y}\right)^{i} \left(\alpha\frac{(1-\alpha)_{(l-1)}}{l!}\right)^{i}\frac{n}{n-il}{n-il+i\alpha-1\choose n-il-1}.
\end{align} 
\end{lem}

In the next theorem, which is the main result of the paper, we exploit \eqref{mgf_prior} and \eqref{prior_rising_1} in order to establish the large deviation principle for $M_{l,n}$. The proof of this theorem is split into three main parts. The first two parts deal with the large deviation principle for $M_{l,n}$ under the assumption $\alpha\in(0,1)$ and $\theta=0$, whereas the third part deals with the general case $\alpha\in(0,1)$ and $\theta>-\alpha$.

\begin{thm}\label{teorema_prior}
For any $\alpha\in(0,1)$ and $\theta>-\alpha$, as $n$ tends to infinity, $n^{-1}M_{l,n}$ satisfies a large deviation principle with speed $n$ and rate function $I_{l}^{\alpha}(x)=\sup_{\lambda}\{\lambda x-\Lambda_{\alpha,l}(\lambda) \}$ where $\Lambda_{\alpha,l}$ is specified in \eqref{s-ldp-eq1}. In particular, for almost all $x>0$ 
\begin{displaymath}
\lim_{n\rightarrow+\infty}\frac{1}{n}\log\mathbb{P}\left[\frac{M_{l,n}}{n}> x\right]=-I^{\alpha}_{l}(x)
\end{displaymath}
where $I^{\alpha}_{l}(0)=0$ and $I^{\alpha}_{l}(x)<+\infty$ for $x\in(0,1/l]$. Moreover $I^{\alpha}_{l}(x)=+\infty$ for $x\notin[0,1/l]$.
\end{thm}

\begin{proof}
In the first part of the proof we show that, assuming $\alpha\in(0,1)$ and $\theta=0$, $n^{-1}M_{l,n}$ satisfies a large deviation principle with speed $n$ and we characterize the corresponding rate function $I_{l}^{\alpha}$. For large $n$, by means of \eqref{prior_rising_2} we have
\begin{equation*}\label{s-eq6}
\E_{\alpha,0}[M_{l,n}]=\frac{\alpha (1-\alpha)_{(l-1)}}{\alpha\Gamma(n)l!}(n)_{[l]}(\alpha)_{(n-l)}\approx  n^{\alpha}
\end{equation*}
and
\begin{equation*}\label{s-eq7}
G_{M_{l,n}}(y;\alpha,0)=\sum_{i=1}^{\lfloor n/l\rfloor}\tilde{y}^i\frac{n}{n-il}{n-il+\alpha i-1\choose n-il-1}\notag
\end{equation*}  
where $\tilde{y}= \alpha y (1-\alpha)_{(l-1)}/(1-y)l!$. If $n/l$ is an integer, then the final term in the above expression is $n\tilde{y}^{n/l}$. By direct calculation we have that $\lim_{n\rightarrow +\infty}n^{-1}\log \E_{\alpha,0}[e^{\lambda M_{l,n}}]=0$ for any $\lambda \leq 0 $. Also, for any $\lambda >0$ and $y=1-e^{-\lambda}$,
\begin{align*}
&\lim_{n\rightarrow +\infty}\frac{1}{n}\log G_{M_{l,n}}(y;\alpha,0)\\
&\quad=  \lim_{n\rightarrow +\infty}\frac{1}{n}\log \max \left\{\tilde{y}^i {n-il+\alpha i-1\choose n-il-1}\text{; }i=0, \ldots, \frac{n}{l}\right\}\notag\\
&\quad=\lim_{n\rightarrow +\infty} \max \left\{\frac{1}{n}\log\tilde{y}^i {n-il+\alpha i-1\choose n-il-1}\text{; }i =0, \ldots, \frac{n}{l}\right\}.\notag
\end{align*} 
For  $\alpha i < 1$, it is clear that $\lim_{n\rightarrow +\infty}n^{-1}\log \tilde{y}^i {n-(l-\alpha)i-1\choose n-li-1} =0$. For $i$ satisfying $0\leq n-il <1$, there are two possibilities: either $n=il$ or $i=\lfloor n/l\rfloor < n/l$. In both cases $\lim_{n\rightarrow +\infty}n^{-1}\log\tilde{y}^i {n-(l-\alpha)i-1\choose n-li-1} =l^{-1}\log \tilde{y}$. Next we consider the case for which $i$ satisfies  $n-li \geq 1$ and $\alpha i \geq 1$. For $0<\epsilon<1/l $, set $\phi(\epsilon)= \epsilon\log\epsilon$ and
$ \varphi(\epsilon)=\phi(1-(l-\alpha)\epsilon)-\phi(1-l\epsilon)-\phi(\alpha\epsilon) +\epsilon\log \tilde{y}$. Using $\Gamma(z)=\sqrt{2\pi}z^{z-1/2}e^{-z}[1+r(z)]$, where we set $|r(z)| \leq e^{1/12z}-1$ for any $z>0$, we have
\begin{align*}
&{n-il+\alpha i-1\choose n-il-1}\\
&\quad=\frac{\Gamma(n-(l-\alpha)i)}{\alpha i\Gamma(n-li)\Gamma(\alpha i)}\\
&\quad= \frac{(1+r(n-(l-\alpha)i))e}{\sqrt{2\pi}(1+r(n-li))(1+r(\alpha i))}\left(\frac{(n-li)}{\alpha i(n-(l-\alpha)i)}\right)^{1/2}\\
&\quad\quad \times \frac{(1-(l-\alpha)i/n)^{n-(l-\alpha)i}}{(1-li/n)^{n-li}(\alpha i/n)^{\alpha i}}\\
&\quad= \frac{(1+r(n-(l-\alpha)i))e}{\sqrt{2\pi}(1+r(n-li))(1+r(\alpha i))}\left(\frac{(n-li)(\alpha i+1)}{(n-(l-\alpha)i)}\right)^{1/2}\\
&\quad\quad \times \exp\left\{n \left[\phi\left(1-(l-\alpha)\frac{i}{n}\right)-\phi\left(1-l\frac{i}{n}\right)-\phi\left(\alpha \frac{i}{n}\right)\right]\right\}.
\end{align*}
The fact that $\alpha^{-1}\leq i \leq (n-1)/l$ implies that $(1+r(n-(l-\alpha)i))e/\sqrt{2\pi}(1+r(n-li))(1+r(\alpha i))$ is uniformly bounded  from above by a constant $d_1>0$. Hence,
\begin{align}\label{eq:teo1_1}
&{n-il+\alpha i-1\choose n-il-1}\\
&\notag\quad \leq d_1 \sqrt{n}  \exp\left\{ n \left[\phi\left(1-(l-\alpha)\frac{i}{n}\right)-\phi\left(1-l\frac{i}{n}\right)-\phi\left(\alpha \frac{i}{n}\right)\right]\right\},
\end{align}
and 
\begin{align}\label{eq:teo1_2}
&\lim_{n\rightarrow+\infty} \max \left\{\frac{1}{n}\log\tilde{y}^i {n-il+\alpha i-1\choose n-il-1}\text{; }\frac{1}{\alpha}\leq i \leq \frac{(n-1)}{l}\right\}\\
&\notag\quad\leq \max\left\{\varphi(\epsilon): 0<\epsilon <\frac{1}{l}\right\}.
\end{align}
In particular, by combining the inequalities stated in \eqref{eq:teo1_1} and \eqref{eq:teo1_2}, respectively, we have
\begin{align*} 
&\lim_{n\rightarrow +\infty}\frac{1}{n}\log G_{M_{l,n}}(y;\alpha,0)\notag\\
&\quad=  \lim_{n\rightarrow+ \infty}\frac{1}{n}\log \max \left\{\tilde{y}^i {n-il+\alpha i-1\choose n-il-1}\text{; }i =0, \ldots, \frac{n}{l}\right\}\notag\\
&\quad =\max \left\{\max \left\{\lim_{n\rightarrow+ \infty}\frac{1}{n}\log\tilde{y}^i {n-il+\alpha i-1\choose n-il-1}\text{; }i <\frac{1}{\alpha}\right\}\right.,\\
&\quad\quad\quad\quad \left.\max \left\{\lim_{n\rightarrow+ \infty}\frac{1}{n}\log\tilde{y}^i {n-il+\alpha i-1\choose n-il-1}\text{; }\frac{1}{\alpha}\leq i\leq \frac{n-1}{l}\right\}\text{, } \frac{1}{l}\log\tilde{y}\right\}\notag\\
&\quad= \max\left\{0, \max\left\{\varphi(\epsilon):0<\epsilon<\frac{1}{l}\right\}\text{, } \frac{1}{l}\log\tilde{y}\right\}\notag\\
&\quad \leq \max\left\{\varphi(\epsilon): 0<  \epsilon < \frac{1}{l}\right\}.\notag
\end{align*}
On the other hand, for any $\epsilon$ in $(0,1/l)$, there exists a sequence $(i_n)_{n\geq1}$ such that $(i_n/n)_{n\geq1}$ converges to $\epsilon$ as $n$ tends to infinity. For this particular sequence
\begin{align*}\label{s-eq9}
\varphi(\epsilon)&= \lim_{n\rightarrow +\infty}\frac{1}{n}\log\tilde{y}^{i_n} {n-i_nl+\alpha i_n-1\choose n-i_nl-1}\\
&\notag\leq \lim_{n\rightarrow +\infty}\frac{1}{n}\log G_{M_{l,n}}(y;\alpha,0).
\end{align*}
Thus
\begin{equation*}\label{s-eq10}
\lim_{n\rightarrow +\infty}\frac{1}{n}\log G_{M_{l,n}}(y;0,\alpha)=\max\left\{\varphi(\epsilon): 0\leq \epsilon \leq\frac{1}{l}\right\}.
\end{equation*}
 Noting that $$\varphi'(\epsilon)=-(l-\alpha)\log (1-(l-\alpha)\epsilon) +l\log(1-l\epsilon)-\alpha\log\alpha\epsilon  +\log\tilde{x},$$ one has $$\varphi(\epsilon)= \log(1-(l-\alpha)\epsilon)-\log(1-l\epsilon)+\varphi'(\epsilon)\epsilon.$$ Moreover, since $\varphi'(0+)=+\infty$ and $\varphi'(1/l-)=-\infty$, then the function $\varphi(\epsilon)$ reaches a maximum at a point $\epsilon_0$ in $(0,1/l)$ where $\varphi'(\epsilon_0)=0$.  Clearly $\epsilon_0$ depends on $\alpha$, $l$ and $\lambda$. Moreover note that $\varphi''(\epsilon)= -\alpha/\epsilon(1-(l-\alpha)\epsilon)(1-l\epsilon)<0$, which implies that $\epsilon_0(\lambda)$ is unique and $\Lambda_{\alpha,l}(\lambda)=   \log[1+\alpha\epsilon_0/(1-l\epsilon_0)]$. In particular, since
\begin{displaymath}
\log \tilde{x}=\lambda +\log\frac{e^{\lambda}-1}{e^{\lambda}}+\log \frac{\alpha (1-\alpha)_{(l-1)}}{l!}
\end{displaymath}
and $\varphi'(\epsilon_0)=-(l-\alpha)\log (1-(l-\alpha)\epsilon_0) +l\log(1-l\epsilon_0)-\alpha\log\alpha\epsilon_0  +\log\tilde{x}=0$, one has
\begin{align*} \label{s-eq12}
&\lambda +\log\frac{e^{\lambda}-1}{e^{\lambda}}+\log \frac{\alpha (1-\alpha)_{(l-1)}}{\alpha^{\alpha}l!}\\
&\quad= l\log\frac{1-(l-\alpha)\epsilon_0}{1-l\epsilon_0} +\alpha \log\frac{\epsilon_0}{1-(l-\alpha)\epsilon_0}.
\end{align*}
Set
\begin{equation}\label{eq:funch1}
h_1(\lambda)=\lambda +\log\frac{e^{\lambda}-1}{e^{\lambda}}+\log \frac{\alpha (1-\alpha)_{(l-1)}}{\alpha^{\alpha}l!}
\end{equation}
and
\begin{equation}\label{eq:funch2}
h_2(\epsilon_0)=  l\log\frac{1-(l-\alpha)\epsilon_0}{1-l\epsilon_0} +\alpha \log\frac{\epsilon_0}{1-(l-\alpha)\epsilon_0}.
\end{equation}
Note that, since both the functions $h_1$ and $h_2$ are strictly increasing functions with differentiable inverses, then $\epsilon_0 = h_2^{-1}\circ h_1(\lambda)$ is a differentiable  strictly increasing function and, in particular, we have $\lim_{\lambda \rightarrow 0}\epsilon_0=0$ and $\lim_{\lambda\rightarrow +\infty}\epsilon_0 =1/l$. Now, if we set $\Lambda_{\alpha,l}(\lambda)$ to be zero for nonpositive $\lambda$, and for $\lambda>0$
\begin{equation}\label{s-ldp-eq1}
\Lambda_{\alpha,l}(\lambda)= \log\left(1+\frac{\alpha h_2^{-1}\circ h_1(\lambda)}{1-l h_2^{-1}\circ h_1(\lambda)}\right),
\end{equation}
then it is clear that $\{\lambda: \Lambda_{\alpha,l}(\lambda)<+ \infty\}=\mathbb{R}$ and $\Lambda_{\alpha,l}(\lambda)$ is differentiable for $\lambda \neq 0$. The left derivative of $\Lambda_{\alpha,l}(\lambda)$ at zero is clearly zero. On the other hand, for $\lambda >0$
\begin{displaymath} 
\frac{d\Lambda_{\alpha,l}(\lambda)}{d \lambda}= \bigg[ \frac{\alpha-l}{1+(\alpha-l)\epsilon_0}+\frac{l}{1-l\epsilon_0}\bigg]\frac{d\epsilon_0}{d\lambda}.
\end{displaymath}
Since $\epsilon_0$ converges to zero  it follows from direct calculation that, as  $\lambda \downarrow 0$ one has
\begin{displaymath}
\frac{d\epsilon_0}{d\lambda}=\frac{(e^{h_1(\lambda)})'}{(e^{h_2(\epsilon)})'|_{\epsilon=\epsilon_0}}\rightarrow 0.
\end{displaymath}
Accordingly $\Lambda_{\alpha,l}(\lambda)$ is differentiable everywhere. By the G\"artner-Ellis theorem (see Dembo and Zeitouni \cite{DZ98} for details), a large deviation principle holds for $n^{-1}M_{l,n}$ on space $\mathbb{R}$ as $n$ tends to infinity  with speed $n$ and good rate function $I^{\alpha}_{l}(x)=\sup_{\lambda}\{\lambda x-\Lambda_{\alpha,l}(\lambda) \}$. This completes the first part of the proof. In the second part of the proof we further specify the rate function $I^{\alpha}_{l}$. In particular,  let us rewrite $\Lambda_{\alpha,l}(\lambda)$ as $\Lambda_{\alpha,l}(\lambda)=  \lambda/l+\tilde{\Lambda}_{\alpha,l}(\lambda)$, where we defined
\begin{displaymath}
\tilde{\Lambda}_{\alpha,l}(\lambda)=-\lambda/l,
\end{displaymath}
for $\lambda\leq 0$, and 
\begin{displaymath}
\tilde{\Lambda}_{\alpha,l}(\lambda)=\frac{1}{l}\log\frac{e^{\lambda}-1}{e^{\lambda}}+\frac{1}{l} \log \frac{\alpha (1-\alpha)_{(l-1)}}{\alpha^{\alpha}l!}-\frac{\alpha}{l} \log\frac{\epsilon_0}{1-(l-\alpha)\epsilon_0}
\end{displaymath}
for $\lambda>0$. Since there exists a strictly positive constant $d_2>0$ such that $\epsilon_0 \geq d_2$ for $\lambda \geq 1$, then $\tilde{\Lambda}_{\alpha,l}$ is uniformly bounded for $\lambda\geq 1$. This implies that the rate function $I_l^{\alpha}(x)=\sup_{\lambda}\{\lambda \left(x-1/l\right)-\tilde{\Lambda}_{\alpha,l}(\lambda)\}$ is infinity for any $x>1/l$, which is consistent with the fact that $n^{-1}M_{l,n}\leq 1/l$. Additionally we have
\begin{equation}\label{rate_prec_1}
I_{l}^{\alpha}(x)=\left\{\begin{array}{ll}
0&\hspace{0.4cm}\text{ if }x=0\\[4pt]
<+\infty&\hspace{0.4cm}\text{ if } x\in(0,1/l]\\[4pt]
+\infty&\hspace{0.4cm}\text{ otherwise }.
\end{array}\right.
\end{equation}
For this to hold, we need to verify  that $I_l^{\alpha}(x)$ is finite for $x$ in $(0,1/l]$.  By definition, 
\begin{equation}\label{s-eq13a}
\sup_{0\leq \lambda \leq 1}\{\lambda x -\Lambda(\lambda)\} \leq  \sup_{0\leq \lambda \leq 1}\{\lambda x\} =x<+\infty
\end{equation}
for any $x$ in $(0,1/l]$.
For any $\lambda\geq 1$, let $d_2$ be the value of $\epsilon_0$ at $\lambda=1$. Then $\epsilon_0 \geq d_2$ for any $\lambda\geq 1$ and this implies that  $\tilde{\Lambda}(\lambda)$ is bounded for all $\lambda \geq 1$. Accordingly, we can write $\sup_{\lambda \geq 1}\{\lambda\left(y-1/l\right)-\tilde{\Lambda}_{\alpha,l}(\lambda)\} \leq  \sup_{\lambda \geq 1}\{|\tilde{\Lambda}_{\alpha,l}(\lambda)|\}<+\infty$, which combined with \eqref{s-eq13a} implies \eqref{rate_prec_1}. This completes the second part of the proof. Finally, in the third part of the proof we extend the large deviation principle to the case $\alpha\in(0,1)$ and $\theta>-\alpha$. By combining the definition \eqref{eq_genfun_prior} with \eqref{prior_rising_1}, and by means of standard combinatorial manipulations, one has
\begin{align*}
&G_{M_{l,n}}(y;\alpha,\theta)\label{s-may23-eq1}\\
&\quad=\sum_{i=0}^{\lfloor{n/l\rfloor}} D(\alpha,\theta,n,i) \left(y \alpha\frac{(1-\alpha)_{(1-y)(l-1)}}{l!}\right)^{i}
\frac{n}{(n-il)}{n-il+i\alpha-1\choose n-il-1}, \notag
\end{align*}
where the function $D$ is such that $D(\alpha,\theta,n,0)=1$ and, for any $1\leq i \leq\lfloor n/l \rfloor$,
 \begin{equation*}\label{s-may27-eq6} 
 D(\alpha,\theta,n,i)= \frac{\Gamma(n)}{(\theta+1)_{(n-1)}}\frac{(\theta/\alpha+1)_{(i-1)}}{\Gamma(i)}\frac{(\theta+i\alpha)_{(n-il)}}{(i\alpha)_{(n-il)}}.
  \end{equation*}
Since $\theta/\alpha >-1$, it follows from basic algebra that one can find positive constants, say $d_3$ and $d_4$, that are independent of $n$ and $i$ and such that it follows
 \begin{equation}\label{s-may27-eq6}
 d_3 n^{-2} \leq D(\alpha,\theta,n,i) \leq d_4 n^{k}
   \end{equation}
 where $k$ is the smallest integer greater than $1+|\theta|+|\theta/\alpha|$. Accordingly, we have
\begin{align*}
&\lim_{n \rightarrow +\infty}\frac{1}{n}\log \E_{\alpha,\theta}[e^{\lambda M_{l,n}}]\\
&\quad= \lim_{n \rightarrow +\infty}\frac{1}{n}\log G_{M_{l,n}}(y;\alpha,\theta)\label{s-may23-eq2}\\
&\quad= \lim_{n \rightarrow +\infty}\frac{1}{n}\log G_{M_{l,n}}(y;\alpha,0)\notag\\
&\quad=\lim_{n \rightarrow +\infty}\frac{1}{n}\log \E_{\alpha,0}[e^{\lambda M_{l,n}}]=\Lambda_{\alpha,l}(\lambda).
\end{align*}
Then, for any $\alpha\in(0,1)$ and $\theta>-\alpha$, $n^{-1}M_{l,n}$ satisfies a large deviation principle with speed $n$ and rate function $I^{\alpha}_{l}$. This completes the third part of the proof.
\end{proof}

In general it is difficult to get a more explicit expression for $I_{l}^{\alpha}$. Indeed, $\Lambda_{\alpha,l}$ depends on $\lambda$ in an implicit form, namely $\Lambda_{\alpha,l}$ is a function of $\lambda$ in terms of $h_2^{-1}\circ h_1(\lambda)$, where where $h_{1}$ and $h_{2}$ are in \eqref{eq:funch1} and  \eqref{eq:funch2} respectively. However, under the assumption $\alpha=1/2$ and $l=1$, an explicit expression for $I^{\alpha}_{l}$ can be derived. For any $\alpha\in(0,1)$ and $\theta>-\alpha$, the rate function $I^{\alpha}_{l}$ displayed in \eqref{s-ldp-eq1} can be easily evaluated by means of standard numerical techniques.   

\begin{prp}\label{proposition_prior} 
Let $B_1$ be the function specified in  \eqref{s-eq17}. Then, for any $x\in [0, 1]$
\begin{displaymath}
I_{1}^{1/2}(x)=x\log [B_1(x)+1]+\log 2 -\log \left(1+\sqrt{B^2_1(x)+1}\right).
\end{displaymath}
\end{prp}

\begin{proof}
Under the assumption $\alpha=1/2$ and $l=1$, the equation $-(l-\alpha)\log (1-(l-\alpha)\epsilon_0) +l\log(1-l\epsilon_0)-\alpha\log\alpha\epsilon_0  +\log\tilde{y}=0$ becomes of the form
\begin{displaymath}
-\frac{1}{2}\log\left(1-\frac{\epsilon_0}{2}\right) +\log(1-\epsilon_0)-\frac{1}{2}\log\epsilon_0  +\frac{1}{2}\log 2 +\log(e^{\lambda}-1) -\log 2=0. 
\end{displaymath}
Equivalently we have $(e^{\lambda}-1)^2 =(2-\epsilon_0)\epsilon_0/(1-\epsilon_0)^2$. Solving the equation we obtain $\epsilon_0 =1-1/\sqrt{B^2+1}$ with $B =e^{\lambda}-1$.  Going back to the rate function, we have
\begin{align*}
I_{1}^{1/2}(x)&= \sup_{\lambda>0}\left\{\lambda x -\log \frac{1-\epsilon_0/2}{1-\epsilon_0}\right\}\\
&= \sup_{\lambda>0}\left\{\lambda x -\log \frac{2-\epsilon_0}{1-\epsilon_0}\right\}+\log 2\\
&= \sup_{\lambda>0}\left\{\lambda x -\log (1+\sqrt{B^2+1})\right\}+\log 2.
\end{align*}
It is known  that $I_{1}^{1/2}(0)=0$. Moreover, for $x=1$,  we have  the following expression
\begin{align*}
&\sup_{\lambda>0}\left\{\lambda  -\log (1+\sqrt{B^2+1})\right\}\\
&\quad= \sup_{\lambda>0}\left\{\log \frac{B+1}{1+\sqrt{B^2+1}}\right\}\\
&\quad= \lim_{\lambda \rightarrow +\infty}\log \frac{B+1}{1+\sqrt{B^2+1}}=0,
\end{align*} 
which implies that $I^{1/2}_{1}(1)=\log2$. In general, for $0<x<1$, set $h(\lambda)= \lambda x - \log(1+\sqrt{B^2+1})$. Then $h'(\lambda)=x- B(B+1)/(B^2 +1 +\sqrt{B^2+1})$ and, in particular, the solution of the equation $h'(\lambda)=0$ satisfies the following identity
\begin{equation}\label{s-eq16}
 (1-x)^2B^3 +2(1-x)B^2 +(1-x)^2B-2x=0,
\end{equation}
and
\begin{align*}
\Delta &= 64 x(1-x)^3 +4(1-x)^6-36x(1-x)^5-4(1-x)^8-108x^2(1-x)^4\\
&=4(1-x)^6[1-(1-x)^2] + (1-x)^3x[64 - 36(1-x)^2 -108x(1-x)] \\
&\geq 4(1-x)^6[1-(1-x)^2] + (1-x)^3x[64-36-27]>0
\end{align*}
is the discriminant. Let $G(B)$ denote the left-hand side of \eqref{s-eq16}. By a direct calculation it follows that $G'(B)=0$ has two negative roots. This, combined with the fact that $G(0)=-2x <0$, implies that one and only one of the three roots of \eqref{s-eq16} is positive. 
Denote this root by $B_1(x)$. Then the rate function is
\begin{equation}\label{rate12}
 I_{1}^{1/2}(x)= x\log [B_1(x)+1]+\log 2 -\log \left(1+\sqrt{B^2_1(x)+1}\right).
\end{equation}
 Making a change of variable  in \eqref{s-eq16} such that $C=B+2/(3(1-x))$ we obtain the following depressed form of the equation $C^3 +pC+q =0$ where $p= 1-4/(3(1-x)^2)<0$ and $q=[16-18(1-x)^2-54 x(1-x)]/(27(1-x)^3)$. Then
 \begin{equation}\label{s-eq17}
 B_1(x)= 2\sqrt{\frac{-p}{3}}\cos\bigg(\frac{1}{3}\arccos \bigg(\frac{3q}{2p}\sqrt{\frac{-3}{p}}\bigg)\bigg) -\frac{2}{3(1-x)}
 \end{equation}
follows by a direct application of the Vi\'ete's trigonometric formula. The proof is completed by combining the rate function \eqref{rate12} with the function $B_{1}$ in \eqref{s-eq17}.
\end{proof}

To some extent Theorem \ref{teorema_prior} provides a generalization of the large deviation principle for $K_{n}$ introduced in Theorem 1.2 of Feng and Hoppe \cite{Fen(98)}, for any $\alpha\in(0,1)$ and $\theta>-\alpha$. Indeed, recall that one has the following relations between $K_{n}$ and $M_{l,n}$: $K_n=\sum_{1\leq i\leq n}M_{l,n}$ and $n=\sum_{1\leq i\leq n}lM_{l,n}$. However, so far it is not clear to us how to relate the large deviation principle for $M_{l,n}$ with  the large deviation principle for $K_n$. In this respect we retain that the results in Dinwoodie and Zabell \cite{Din(92)} may be helpful  in understanding such a relation.


\section{Conditional large deviations}\label{sec3}

For any $\alpha\in(0,1)$ and $\theta>-\alpha$, let $(X_1, \ldots,X_n)$ be an initial sample from  $\tilde{P}_{\alpha,\theta,\nu}$ and let $(X_{n+1},\ldots,X_{n+m})$ be an additional sample, for any $m\geq1$. Furthermore, let $X^{\ast}_{1},\ldots,X^{\ast}_{K_{n}}$ be the labels identifying the $K_{n}$ blocks generated by $(X_{1},\ldots,X_{n})$ with corresponding frequencies $\mathbf{N}_{n}$, and let $L_{m}^{(n)}=\sum_{1\leq i\leq m}\prod_{1\leq k\leq K_{n}}\mathbbm{1}_{\{X_{k}^{\ast}\}^{c}}(X_{n+i})$ be the number of elements in the additional sample that do not coincide with elements in the initial sample. If we denote  by $K_{m}^{(n)}$ the number of new blocks generated by these $L_{m}^{(n)}$ elements and by $X^{\ast}_{K_{n}+1},\ldots,X^{\ast}_{K_{n}+K_{m}^{(n)}}$ their labels, then
\begin{equation}\label{eq:new_freq}
S_{i}=\sum_{l=1}^{m}\mathbbm{1}_{\{X^{\ast}_{K_{n}+i}\}}(X_{n+l}),
\end{equation}
for $i=1,\ldots,K_{m}^{(n)}$, are the frequencies of the $K_{m}^{(n)}$ blocks. The frequencies of the blocks generated by the remaining $m-L_{m}^{(n)}$ elements of the additional sample are
\begin{equation}\label{eq:old_freq}
R_{i}=\sum_{l=1}^{m}\mathbbm{1}_{\{X^{\ast}_{i}\}}(X_{n+l}),
\end{equation}
for $i=1,\ldots,K_{n}$. The blocks generated by the $m-L_{m}^{(n)}$ elements of the additional sample are termed ``old"  to be distinguished from the $K_{m}^{(n)}$ new blocks generated by the $L_{m}^{(n)}$ elements of the additional sample. The random variables \eqref{eq:new_freq} and \eqref{eq:old_freq}, together with $L_{m}^{(n)}$ and $K_{m}^{(n)}$, completely describe the conditional random partition induced by $(X_{n+1},\ldots,X_{n+m})$ given $(X_{1},\ldots,X_{n})$. See Lijoi et al. \cite{Lij(08)} and Favaro et al. \cite{Fav(13)} for a comprehensive study on the conditional distributions of these random variables given the initial sample.

The random variables \eqref{eq:new_freq} and \eqref{eq:old_freq} lead to define the number $M_{l,m}^{(n)}$ of blocks with frequency $l$ in the enlarged sample $(X_{1},\ldots,X_{n+m})$. This is the number of new blocks with frequency $l$ generated by $(X_{n+1},\ldots,X_{n+m})$ plus the number of old blocks with frequency $l$ that arise by updating, via $(X_{n+1},\ldots,X_{n+m})$, the frequencies already induced by $(X_{1},\ldots,X_{n})$. Specifically, let
\begin{equation*}\label{eq:freq_new_block}
N_{l,m}^{(n)}=\sum_{i=1}^{K_{m}^{(n)}}\mathbbm{1}_{\{S_{i}=l\}}
\end{equation*}
be the number of new blocks with frequency $l$. Specifically, these new blocks are generated by the $L_{m}^{(n)}$ elements of the additional sample. Furthermore, let
\begin{equation*}\label{eq:freq_old_block}
O_{l,m}^{(n)}=\sum_{i=1}^{K_{n}}\mathbbm{1}_{\{N_{i}+R_{i}=l\}}
\end{equation*}
be the number of old blocks with frequency $l$. Specifically, these old blocks are generated by updating, via the $m-L_{m}^{(n)}$ elements of the additional sample,  the frequencies of random partition induced by the initial sample. Therefore, $M_{l,m}^{(n)}=O_{l,m}^{(n)}+N_{l,m}^{(n)}$. The conditional distribution of $M_{l,m}^{(n)}$, given the initial sample, has been recently derived and investigated in Favaro et al. \cite{Fav(13)}. Hereafter we present a large deviation principle, as $m$ tends to infinity, for $M_{l,m}^{(n)}\,|\,(K_{n},\mathbf{N}_{n})$.

The study of large deviations for $M_{l,m}^{(n)}\,|\,(K_{n},\mathbf{N}_{n})$ reduces to the study of large deviations for the conditional number of new blocks with frequency $l$, namely $N_{l,m}^{(n)}\,|\,(K_{n},\mathbf{N}_{n})$. Indeed $N_{l,m}^{(n)}\leq M_{l,m}^{(n)}\leq N_{l,m}^{(n)}+n$. Hence, by means of a direct application of Corollary B.9 in Feng \cite{Feng10}, the quantities $m^{-1}M_{l,m}^{(n)}\,|\,(K_{n},\mathbf{N}_{n})$ and $m^{-1}N_{l,m}^{(n)}\,|\,(K_{n},\mathbf{N}_{n})$ satisfy the same large deviation principle. This large deviation principle is established through the study of the moment generating function of $N_{l,m}^{(n)}\,|\,(K_{n},\mathbf{N}_{n})$. For any $\lambda>0$ and  $y=1-\text{e}^{-\lambda}$, let
\begin{align}\label{eq_genfun_posterior}
&G_{N_{l,m}^{(n)}}(y;\alpha,\theta)\\
&\notag\quad=\E_{\alpha,\theta}\left[\left(\frac{1}{1-y}\right)^{N_{l,m}^{(n)}}\,|\,K_{n}=j,\mathbf{N}_{n}=\mathbf{n}\right]\\
&\notag\quad=\sum_{i\geq0}\frac{y^{i}}{i!}\mathbb{E}_{\alpha,\theta}[(N_{l,m}^{(n)})_{(i)}\,|\,K_{n}=j,\mathbf{N}_{n}=\mathbf{n}].
\end{align}
Theorem 1 in Favaro et al. \cite{Fav(13)} provides an explicit expression for the falling factorial moment $\mathbb{E}_{\alpha,\theta}[(N_{l,m}^{(n)})_{[r]}\,|\,K_{n}=j,\mathbf{N}_{n}=\mathbf{n}]$. Hence, by exploiting the aforementioned relation between falling factorials and rising factorials, an explicit expression for $\mathbb{E}_{\alpha,\theta}[(N_{l,m}^{(n)})_{(r)}\,|\,K_{n}=j,\mathbf{N}_{n}=\mathbf{n}]$ is obtained. Specifically,
\begin{align}\label{fat_post_alpha}
&\E_{\alpha,\theta}[(N_{l,m}^{(n)})_{(r)}\,|\,K_{n}=j,\mathbf{N}_{n}=\mathbf{n}]\\
&\notag\quad=r!\sum_{i=0}^{r}{r-1\choose r-i}\frac{\left(\frac{\alpha(1-\alpha)_{(l-1)}}{l!}\right)^{i}\left(\frac{\theta}{\alpha}\right)_{(j+i)}(m)_{[il]}(\theta+i\alpha+n)_{(m-il)}}{i!(\theta+n)_{(m)}(\theta/\alpha)_{(j)}}
\end{align}
and
\begin{align}\label{fat_post_0}
&\E_{\alpha,0}[(N_{l,m}^{(n)})_{(r)}\,|\,K_{n}=j,\mathbf{N}_{n}=\mathbf{n}]\\
&\notag\quad=j(r-1)! \sum_{i=0}^{r}{r\choose i}{j+i-1\choose i-1}\frac{\left(\frac{\alpha(1-\alpha)_{(l-1)}}{l!}\right)^{i}(m)_{[il]}(i\alpha+n)_{(m-il)}}{(n)_{(m)}}
\end{align}
where the sums over $i$ is nonnull for $0\leq i\leq \min(r,\lfloor{m/l\rfloor})$. Note that $\mathbb{E}_{\alpha,\theta}[(N_{l,m}^{(n)})_{(r)}\,|\,K_{n}=j,\mathbf{N}_{n}=\mathbf{n}]=\E_{\alpha,\theta}[(N_{l,m}^{(n)})_{(r)}\,|\,K_{n}=j]$. In other terms the number $K_{n}$ of blocks in the initial sample is a sufficient statistics for $\mathbb{E}_{\alpha,\theta}[(N_{l,m}^{(n)})_{(r)}\,|\,K_{n}=j,\mathbf{N}_{n}=\mathbf{n}]$. This property of sufficiency was pointed out in Favaro et al. \cite{Fav(13)}. Along lines similar to Lemma \ref{lemma_prior}, in the next lemma we provide an explicit expression for the moment generating function $G_{N_{l,m}^{(n)}}(y;\alpha,0)$.

\begin{lem}\label{lemma_post_0}
For any $\alpha\in(0,1)$
\begin{align}\label{mom_gen_post_0}
&G_{N_{l,m}^{(n)}}(y;\alpha,0)\\
&\notag\quad =\frac{m!}{(n)_{(m)}}\sum_{i=0}^{\lfloor m/l\rfloor}\left(\frac{y}{1-y}\right)^{i}\left(\frac{\alpha(1-\alpha)_{(l-1)}}{l!}\right)^{i}\\
&\notag\quad\quad\times{j+i-1\choose i}\frac{(i\alpha+n)}{(m-il)}{n+m+i\alpha-il-1\choose m-il-1}.
\end{align}
\end{lem}

In the next theorem we exploit the moment generating function \eqref{mom_gen_post_0} and the rising factorial moment \eqref{fat_post_alpha} in order to establish the large deviation principle for $M_{l,m}^{(n)}\,|\,(K_{n},\mathbf{N}_{n})$. Such a result provides a conditional counterpart of Theorem \ref{teorema_prior}.

\begin{thm}\label{teorema_posterior}
For any $\alpha\in(0,1)$ and $\theta>-\alpha$, as $m$ tends to infinity, $m^{-1}M_{l,m}^{(n)}\,|\,(K_{n},\mathbf{N}_{n})$ satisfies a large deviation principle with speed $m$ and rate function $I_{l}^{\alpha}(x)=\sup_{\lambda}\{\lambda x-\Lambda_{\alpha,l}(\lambda) \}$ where $\Lambda_{\alpha,l}$ is specified in \eqref{s-ldp-eq1}. In particular, for almost all $x>0$ 
\begin{displaymath}
\lim_{m\rightarrow+\infty}\frac{1}{m}\log\mathbb{P}\left[\frac{M_{l,m}^{(n)}}{m}> x\,|\,K_{n}=j,\mathbf{N}_{n}=\mathbf{n}\right]=-I^{\alpha}_{l}(x)
\end{displaymath}
where $I^{\alpha}_{l}(0)=0$ and $I^{\alpha}_{l}(x)<+\infty$ for $x\in(0,1/l]$. Moreover $I^{\alpha}_{l}(x)=+\infty$ for $x\notin[0,1/l]$.
\end{thm}

\begin{proof}
As we anticipated, in order to prove the theorem, it is sufficient to prove the large deviation principle for $N_{l,m}^{(n)}\,|\,K_{n}$, for any $\alpha\in(0,1)$ and $\theta>-\alpha$. We start with the assumption $\alpha\in(0,1)$ and $\theta=0$ and then we consider the general case. From the moment generating function \eqref{mom_gen_post_0} we can write
\begin{displaymath}
G_{N_{l,m}^{(n)}} (y;\alpha,0)=\sum_{i=0}^{\lfloor m/l\rfloor}\tilde{y}^i C(i,m; n,j,\alpha,l)
\end{displaymath}
where
\begin{align*}
&C(i,m;n,j,\alpha,l)\\
&\quad=\frac{m!}{(n)_{(m)}}{j+i-1\choose i}\frac{i\alpha+n}{m-il}{n+m+i\alpha-il-1\choose m-il-1}\\
&\quad={n+m +i\alpha -il-1\choose n+m-il-1}\frac{m!}{(n)_{(m)}}{j+i-1\choose i} \frac{(m-il+1)_{(n-2)}}{(i\alpha+1)_{(n-2)}}\\
&\quad={n+m +i\alpha -il-1\choose n+m-il-1}\\
&\quad\quad\times\frac{(n-1)!}{(m+1)\cdots(m+n-1)}\frac{(m-il+1)_{(n-2)}}{(i\alpha+1)_{(n-2)}} {j+i-1\choose i}  
\end{align*}
which is bounded below by $((n-1)!/(m+n)^{n-1})^2$, and from above by $(m+n)^{n+j-1}$.  Hence,
\begin{equation}\label{s-may27-eq1}
G_{N_{l,m}^{(n)}}(y;\alpha,0)\leq (m+n)^{n+j-1}G_{M_{l,n+m}}(y;\alpha,0)
\end{equation}
and
\begin{equation}\label{s-may27-eq2}
G_{N_{l,m}^{(n)}}(y;\alpha,0)\geq\frac{\left(G_{M_{l,n+m}}(y;\alpha,0)-\sum_{i=\lfloor m/l\rfloor+1}^{\lfloor (n+m)/l \rfloor}\tilde{y}^i {n+m +i\alpha -il-1\choose n+m-il-1}\right)}{\left(\frac{(n-1)!}{(m+n)^{n-1}}\right)^{-2}}.
\end{equation}
Note that, for any index $i$ such that $\lfloor m/l\rfloor+1\leq i \leq \lfloor(m+n)/l\rfloor$, we can write the following inequalities $1\leq{n+m +i\alpha -il-1\choose n+m-il-1}= (n+m-il)\cdots (n+m-il-1 +i\alpha)/(i\alpha)!\leq (n+1)\cdots (n +i\alpha)/(i\alpha)!\leq (2n+m)^n$ and, in particular one has
\begin{equation}\label{s-may27-eq3}
\lim_{m\rightarrow +\infty}\frac{1}{m}\log \sum_{i=\lfloor m/l\rfloor+1}^{\lfloor (n+m)/l \rfloor}\tilde{y}^i {n+m +i\alpha -il-1\choose n+m-il-1}=0. 
\end{equation}
Accordingly, putting together \eqref{s-may27-eq1}, \eqref{s-may27-eq2} and \eqref{s-may27-eq3}, we obtain the following identity
\begin{align*}
&\lim_{m\rightarrow+ \infty}\frac{1}{m}\log G_{N_{l,m}^{(n)}}(y;\alpha,0)= \lim_{m\rightarrow+\infty}\frac{1}{n+m}\log G_{M_{l,n+m}}(y;\alpha,0)  
\end{align*}
which, once combined with Theorem \ref{teorema_prior}, implies that $m^{-1}N_{l,m}^{(n)}\,|\,K_{n}$ satisfies a large deviation principle with speed $m$ and rate function $I_{l}^{\alpha}$. In order to deal with the general case $\alpha\in(0,1)$ and $\theta>-\alpha$, one needs a term wise comparison between \eqref{fat_post_alpha} and \eqref{fat_post_0}. In particular, for any $i\leq m/l$ let us define
\begin{equation*}\label{s-may27-eq7}
D(m,i;\alpha,\theta,n,j)=\frac{(n)_{(m)}}{(\theta+n)_{(m)}} \frac{(j-1)!(\frac{\theta}{\alpha})_{(j+i)}}{(j+i-1)!(\frac{\theta}{\alpha})_{(j)}}\frac{(\theta +n+i\alpha)_{(m-il)}}{(n+i\alpha)_{(m-il) 1}}.
\end{equation*}
Then, one has
\begin{align*}
&\E_{\alpha,\theta}[(N_{l,m}^{(n)})_{(r)}\,|\,K_{n}=j]\\
&\notag\quad=\frac{j}{(n)_{(m)}}(r-1)! \sum_{i=0}^{r}D(m,i;\alpha,\theta,n,j){r\choose i}{j+i-1\choose i-1}(m)_{[il]}\\
&\notag\quad\quad\times\left(\frac{\alpha(1-\alpha)_{(l-1)}}{l!}\right)^{i}(i\alpha+n)_{(m-il)}.
\end{align*}
By an argument similar to those used in deriving \eqref{s-may27-eq6} it follows that one can find constants $d_5>0$ and $d_6>0$ and positive integers $k_1$ and $k_2$ independent of $m$ and $i$ such that $d_5 (n+m)^{-k_1}\leq D(m,i;\alpha,\theta,n,j)\leq d_6 (n+m)^{k_2}$ which leads to
\begin{align*}
&d_5 \left(\frac{1}{n+m}\right)^{k_1}G_{N_{l,m}^{(n)}}(y;\alpha,0)\\
&\quad\leq G_{N_{l,m}^{(n)}}(y;\alpha,\theta)\\
&\quad\leq  G_{N_{l,m}^{(n)}}(y;\alpha,0) d_6 (n+m)^{k_2}.
\end{align*}
Such a result, combined with Theorem \ref{teorema_prior}, implies that $m^{-1}N_{l,m}^{(n)}\,|\,K_{n}$ satisfies a large deviation principle with speed $m$ and rate function $I_{l}^{\alpha}$. Hence, by a direct application of Corollary B.9 in Feng \cite{Feng10}, $m^{-1}M_{l,m}^{(n)}\,|\,(K_{n},\mathbf{N}_{n})$ satisfies a large deviation principle with speed $m$ and rate function $I_{l}^{\alpha}$, and the proof is completed.
\end{proof}

In contrast with the fluctuations \eqref{eq:asim_prior_2pd_freq} and \eqref{eq:fluct_post_freq}, Theorem \ref{teorema_prior} and Theorem \ref{teorema_posterior} show that in terms of large deviations the given initial sample $(X_{1},\ldots,X_{n})$ have no long lasting impact. Specifically the large deviation principles for $M_{l,n}$ and $M_{l,m}^{(n)}\,|\,(K_{n},\mathbf{N}_{n})$ are equivalent when $n$ and $m$ tend to infinity, respectively. This is caused by the two different scalings involved, namely $m^{-1}$ for large deviations and $m^{-\alpha}$ for the fluctuations. According to Corollary 20 in Pitman \cite{Pit(96a)},  the initial sample $(X_{1},\ldots,X_{n})$ leads to modify the parameter $\theta$ in the conditional distribution of $\tilde{P}_{\alpha,\theta,\nu}$ given $(X_{1},\ldots,X_{n})$. Hence we conjecture that the conditional and the unconditional large deviation results will be different if $n$ is allowed to grow and leads to large parameter $\theta$. In the unconditional setting this kind of asymptotic behaviour is discussed in Feng \cite{Feng(07)}, where the parameter $\theta$ and the sample size $n$ grow together and the large deviation result will depend on the relative growth rate between $n$ and $\theta$.

If $m$ depends on $n$ and both approach infinity then one can expect very different behaviours in terms of law of large numbers and fluctuations. The large deviation principle for $M_{l,m}^{(n)}\,|\,(K_{n},\mathbf{N}_{n})$ may not be easily derived, by means of a direct comparison argument, from the large deviation principle of $N_{l,m}^{(n)}\,|\,K_{n}$.  In this respect, it is helpful to study the moment generating function
\begin{align}\label{eq:mom_gen_totale}
&G_{M_{l,m}^{(n)}}(y;\alpha,\theta)\\
&\notag\quad=\E_{\alpha,\theta}\left[\left(\frac{1}{1-y}\right)^{M_{l,m}^{(n)}}\,|\,K_{n}=j,\mathbf{N}_{n}=\mathbf{n}\right]\\
&\notag\quad=\sum_{i\geq0}\frac{y^{i}}{i!}\mathbb{E}_{\alpha,\theta}[(M_{l,m}^{(n)})_{(i)}\,|\,K_{n}=j,\mathbf{N}_{n}=\mathbf{n}].
\end{align}
We intend to pursue this study further in a subsequent project. As in Lemma \eqref{lemma_post_0}, an explicit expression for \eqref{eq:mom_gen_totale} follows by combining the rising factorial moments of $M_{l,m}^{(n)}\,|\,(K_{n},\mathbf{N}_{n})$ with the series expansion on the right-hand side of \eqref{eq_genfun_posterior}, and by means of standard combinatorial manipulations. The rising factorial moments of $M_{l,m}^{(n)}\,|\,(K_{n},\mathbf{N}_{n})$ are obtained from Theorem 3 in Favaro et al. \cite{Fav(13)}.


\section{Discussion}
Our large deviation results contribute to the study of conditional and unconditional properties of the Ewens-Pitman sampling model. Theorem \ref{teorema_posterior} has potential applications in the context of Bayesian nonparametric inference for species sampling problems. Indeed, as we pointed out in the Introduction, in such a context $\P[M_{l,m}^{(n)}\in \cdot\,|\,K_{n}=j,\mathbf{N}_{n}=\mathbf{n}]$ takes on the interpretation of the posterior distribution of the number of species with frequency $l$ in a sample $(X_{1},\ldots,X_{n+m})$ from $\tilde{P}_{\alpha,\theta,\nu}$, given the initial observed sample $(X_{1},\ldots,X_{n})$ featuring $K_{n}=j$ species with corresponding frequencies $\mathbf{N}_{n}=\mathbf{n}$. The reader is referred to Favaro et al. \cite{Fav(13)} for a comprehensive account on this posterior distribution with applications to Bayesian nonparametric inference for the so-called rare, or local, species variety.

For large $m$, $m^{-1}M_{l,m}^{(n)}$ is the random proportion of species with frequency $l$ in $(X_{1},\ldots,X_{n+m})$. In  Theorem \ref{teorema_posterior} we characterized the rate function $I^{\alpha}_{l}$ of a conditional large deviation principle associated to such a random proportion. The rate function $I_{l}^{\alpha}$ is nondecreasing over the set $[0,1/l]$. Then the number of discontinuous points of $I_{l}^{\alpha}$ is at most countable and therefore $\inf_{z\geq x}I_{l}^{\alpha}(z)=\inf_{z>x} I_{l}^{\alpha}(z)$ for almost all  $x \in [0,1/l]$. Hence, for almost all $x>0$,
\begin{align}\label{eq1_discuss}
&\lim_{m\rightarrow+\infty}\frac{1}{m}\log\P\left[\frac{M_{l,m}^{(n)}}{m}\geq x\,|\,K_{n}=j,\mathbf{N}_{n}=\mathbf{n}\right]\\
&\notag\quad=\lim_{m\rightarrow+\infty}\frac{1}{m}\log\P\left[\frac{M_{l,m}^{(n)}}{m}> x\,|\,K_{n}=j,\mathbf{N}_{n}=\mathbf{n}\right]=-I_{l}^{\alpha}(x).
\end{align}
Therefore identity \eqref{eq1_discuss} provides a large $m$ approximation of the Bayesian nonparametric estimator $\P[m^{-1}M_{l,m}^{(n)}\geq x\,|\,K_{n}=j,\mathbf{N}_{n}=\mathbf{n}]$, for any $x\geq0$, namely
\begin{equation}\label{eq:tail_est}
\mathcal{T}_{l,m}^{(n)}(x)=\P\left[\frac{M_{l,m}^{(n)}}{m}\geq x \,|\,K_{n}=j,\mathbf{N}_{n}=\mathbf{n}\right]\approx\exp\{-mI_{l}^{\alpha}(x)\}.
\end{equation}
Hereafter we thoroughly discuss $\mathcal{T}_{1,m}^{(n)}$ in Bayesian nonparametric inference for discovery probabilities. In particular we introduce a novel approximation, for large $m$, of the posterior distribution of the probability of discovering a new species at the $(n+m+1)$-th draw. Such an approximation, then, induces a natural interpretation of $\mathcal{T}_{1,m}^{(n)}$ in the context of Bayesian nonparametric inference for the probability of discovering a new species at the $(n+m+1)$-th draw.

\subsection{Discovery probabilities and large deviations}
Let $D_{m}^{(n)}$ be the probability of discovering a new species at the $(n+m+1)$-th draw. Since the additional sample $(X_{n+1}\ldots,X_{n+m})$ is not observed, $D_{m}^{(n)}\,|\,(K_{n},\mathbf{N}_{n})$ is a random probability. The randomness being determined by $(X_{n+1}\ldots,X_{n+m})$. In particular, by means of the predictive distribution  \eqref{predict}, we observe that $\P[D_{m}^{(n)}\in\cdot\,|\,K_{n}=j,\mathbf{N}_{n}=\mathbf{n}]$ is related to $\P[K_{m}^{(n)}\in\cdot\,|\,K_{n}=j,\mathbf{N}_{n}=\mathbf{n}]$ as follows
\begin{align}\label{rand_disc}
&\P[D_{m}^{(n)}\in\cdot\,|\,K_{n}=j,\mathbf{N}_{n}=\mathbf{n}]\\
&\notag\quad=\P[D_{m}^{(n)}\in\cdot\,|\,K_{n}=j]\\
&\notag\quad=\P\left[\frac{\theta+j\alpha+K_{m}^{(n)}\alpha}{\theta+n+m}\in\cdot\,|\,K_{n}=j\right]\\
&\notag\quad=\P\left[\frac{\theta+j\alpha+K_{m}^{(n)}\alpha}{\theta+n+m}\in\cdot\,|\,K_{n}=j,\mathbf{N}_{n}=\mathbf{n}\right],
\end{align}
where the conditional, or posterior, distribution $\P[K_{m}^{(n)}\in \cdot\,|\,K_{n}=j]$ was obtained in Lijoi et al. \cite{Lij(07)} and then investigated in Favaro et al. \cite{Fav(09)}. Specifically, let $\mathscr{C}(n,x,a,b)=(x!)^{-1}\sum_{0\leq i\leq x}(-1)^{i}{x\choose i}(-ia-b)_{(n)}$ be the noncentral generalized factorial coefficient. See Charalambides \cite{Cha(05)} for details. Then, for any $k=0,1,\ldots,m$, 
\begin{equation}\label{post_dist_k}
\P[K_{m}^{(n)}=k\,|\,K_{n}=j]=\frac{(\theta/\alpha+j)_{(k)}}{(\theta+n)_{(m)}}\mathscr{C}(m,k;\alpha,-n+\alpha j),
\end{equation}
and
\begin{equation}\label{estim_dist}
\E_{\alpha,\theta}[K_{m}^{(n)}\,|\,K_{n}=j]=\left(\frac{\theta}{\alpha}+j\right)\left(\frac{(\theta+n+\alpha)_{(m)}}{(\theta+n)_{(m)}}-1\right).
\end{equation}
The distribution \eqref{rand_disc} is the posterior distribution of the probability of discovering a new species at the $(n+m+1)$-th draw. An explicit  expression for this distribution is obtained by means of \eqref{post_dist_k}. Also, $\mathcal{D}_{m}^{(n)}=\E_{\alpha,\theta}[D_{m}^{(n)}\,|\,K_{n}=j]$ is the Bayesian nonparametric estimator, with respect to a squared loss function, of the probability of discovering a new species at the $(n+m+1)$-th draw. An explicit expression of this estimator is obtained by combining \eqref{rand_disc} with \eqref{estim_dist}.

We introduce a large $m$ approximation of $\P[D_{m}^{(n)}\in\cdot\,|\,K_{n}=j]$ and a corresponding large $m$ approximation of the Bayesian nonparametric estimator $\mathcal{D}_{m}^{(n)}$. This approximation sets a novel connection between the posterior distribution of the proportion of species with frequency $1$ in the enlarged sample and the posterior distribution $\P[D_{m}^{(n)}\in\cdot\,|\,K_{n}=j]$. Specifically, by simply combining the fluctuation limit \eqref{eq:fluct_post} with \eqref{rand_disc}, one has the following fluctuation
\begin{equation}\label{eq:fluct_post_discov}
\lim_{m\rightarrow+\infty}\frac{D_{m}^{(n)}}{m^{\alpha-1}}\,|\,(K_{n}=j)=\alpha S_{\alpha,\theta}^{(n,j)}\qquad\text{a.s.}
\end{equation}
where $S_{\alpha,\theta}^{(n,j)}$ has been defined in \eqref{eq:fluct_post} and \eqref{eq:fluct_post_freq}. In particular $\E[S_{\alpha,\theta}^{(n,j)}]=(j+\theta/\alpha)\Gamma(\theta+n)/\Gamma(\theta+n+\alpha)$. Then, for large $m$, the fluctuations \eqref{eq:fluct_post_discov} and \eqref{eq:fluct_post_freq} lead to
\begin{align}\label{eq:approx}
&\P[D_{m}^{(n)}\in\cdot\,|\,K_{n}=j]\\
&\notag\quad\approx\P[m^{\alpha-1}\alpha S_{\alpha,\theta}^{(n,j)}\in\cdot\,|\,K_{n}=j]\\
&\notag\quad\approx\P\left[\frac{M_{1,m}^{(n)}}{m}\in\cdot\,|\,K_{n}=j,\mathbf{N}_{n}=\mathbf{n}\right] 
\end{align}
and
\begin{align}\label{eq4_discuss}
\mathcal{D}_{m}^{(n)}&=\frac{\theta+j\alpha}{\theta+n}\frac{(\theta+n+\alpha)_{m}}{(\theta+n+1)_{m}}\\
&\notag\approx m^{\alpha-1}(j\alpha+\theta)\frac{\Gamma(\theta+n)}{\Gamma(\theta+n+\alpha)}\\
&\notag\approx\E_{\alpha,\theta}\left[\frac{M_{1,m}^{(n)}}{m}\,|\,K_{n}=j,\mathbf{N}_{n}=\mathbf{n}\right]\\
&\notag=\frac{m_{1}}{m}\frac{(\theta+n-1+\alpha)_{m}}{(\theta+n)_{m}}+(\theta+j\alpha)\frac{(\theta+n+\alpha)_{m-1}}{(\theta+n)_{m}}
\end{align}
where the last identity of \eqref{eq4_discuss} is obtained by means of Theorem 3 in Favaro et al. \cite{Fav(13)}. Interestingly, the second approximation of \eqref{eq4_discuss} is somehow reminiscent of the celebrated Good-Turing estimators introduced in Good \cite{Goo(53)} and Good and Toulmin \cite{Goo(56)}. Indeed, it shows that the estimator of the probability of discovering a new species at the $(n+m+1)$-th draw is related to the estimator of the number of species with frequency $1$ in the enlarged sample.

Intuitively, when $\theta$ and $n$ are moderately large and not overwhelmingly smaller than $m$, the exact value of $\mathcal{D}_{m}^{(n)}$ given in \eqref{eq4_discuss} is much smaller than its asymptotic approximation, which is much smaller than the exact value of $m^{-1}\mathcal{M}_{1,m}^{(n)}=\E_{\alpha,\theta}[m^{-1}M_{1,m}^{(n)}\,|\,K_{n}=j,\mathbf{N}_{n}=\mathbf{n}]$. This suggests that a finer normalization constant than $m^{\alpha}$ is to be used in the fluctuations \eqref{eq:fluct_post_freq} and \eqref{eq:fluct_post_discov}, respectively. Equivalent, though less rough, normalization rates for  \eqref{eq:fluct_post_freq} and \eqref{eq:fluct_post_discov} are
\begin{equation}\label{rate_corr1}
r_{M}(m;\alpha,\theta,n,j,m_{1})=\frac{\Gamma(\theta+\alpha+n+m-1)}{\Gamma(\theta+n+m)}\left(m_{1}\frac{\theta+\alpha+n-1}{\theta+j\alpha}+m\right),
\end{equation}
and
\begin{equation}\label{rate_corr2}
r_{D}(m;\alpha,\theta,n,j)=\frac{\Gamma(\theta+\alpha+n+m)}{\Gamma(\theta+n+m+1)}
\end{equation}
respectively. Obviously, in terms of asymptotic $r_{M}(m;\alpha,\theta,n,j,m_{1})/m^{\alpha}\rightarrow1$ and $r_{D}(m;\alpha,\theta,n,j)/m^{\alpha-1}\rightarrow1$ as $m$ tends to infinity. These corrected normalization rates are determined in such a way that $\mathcal{D}_{m}^{(n)}$ and $m^{-1}\mathcal{M}_{1,m}^{(n)}$ coincide with the corresponding asymptotic moments. Of course different procedures may be considered. Note that the number $j$ of species and the number $m_{1}$ of species with frequency $1$ affect the corrected normalization rate \eqref{rate_corr1}. 

Besides being an interesting large $m$ approximation of $\P[D_{m}^{(n)}\in\cdot\,|\,K_{n}=j]$, the result displayed in \eqref{eq:approx} induces a natural interpretation of the conditional large deviation principle of Theorem \ref{teorema_posterior}, with $l=1$, in the context of Bayesian nonparametric inference for discovery probabilities. Indeed by combining the approximations in \eqref{eq:tail_est} and \eqref{eq:approx} we can write the large $m$ approximation
\begin{align}\label{eq5_discuss11}
\mathcal{D}_{m}^{(n)}(x)&=\P[D_{m}^{(n)}\geq x\,|\,K_{n}=j]\\
&\notag\approx\mathcal{T}_{1,m}^{(n)}(x) \\
&\notag\approx\exp\{-mI_{1}^{\alpha}(x)\}.
\end{align}
By exploiting the corrected normalization rates \eqref{rate_corr1} and \eqref{rate_corr1}, a corrected version of \eqref{eq5_discuss11} is
\begin{align}\label{eq5_discuss12}
\mathcal{D}_{m}^{(n)}(x)&=\P[D_{m}^{(n)}\geq x\,|\,K_{n}=j]\\
&\notag\approx\mathcal{T}_{1,m}^{(n)}\left(x\frac{r_{M}(m;\alpha,\theta,n,j,m_{1})}{mr_{D}(m;\alpha,\theta,n,j)}\right)\\
&\notag\approx\exp\left\{-mI_{1}^{\alpha}\left(x\frac{r_{M}(m;\alpha,\theta,n,j,m_{1})}{mr_{D}(m;\alpha,\theta,n,j)}\right)\right\}.
\end{align}
In other terms Theorem \ref{teorema_posterior} with $l=1$ provides a large $m$ approximation of the  Bayesian nonparametric estimator of the right tail of the probability of discovering a new species at the $(n+m+1)$-th draw, without observing $(X_{n+1},\ldots,X_{n+m})$. We point out that If $\alpha=1/2$ then the rate function in the approximations \eqref{eq5_discuss11} and \eqref{eq5_discuss12} can be exactly computed by means of Proposition \ref{proposition_prior}.


\subsection{Illustration}

We present an illustration of our results dealing with a well-known benchmark Expressed Sequence Tag (EST) dataset. This dataset is obtained by sequencing two cDNA libraries of the amitochondriate protist Mastigamoeba balamuthi: the first library is non-normalized, whereas the second library is normalized, namely it undergoes a normalization protocol which aims at making the frequencies of genes in the library more uniform so to increase the discovery rate. See Susko and Roger \cite{Sus(04)} for comprehensive account on the Mastigamoeba cDNA library. For the Mastigamoeba non-normalized the observed sample consists of $n=715$ ESTs with $j=460$ distinct genes whose frequencies are $m_{i,715}=378, 33, 21, 9, 6, 1, 3, 1, 1, 1, 1, 5$ with $i\in\{1,2,\ldots,10\}\cup\{13,15\}$. For the the Mastigamoeba normalized the observed sample consists of $n=363$ with $j=248$ distinct genes whose frequencies are $m_{i,363}=200, 21, 14, 4, 3, 3, 1, 0, 1, 1$ with $i\in\{1,2,\ldots,9\}\cup\{14\}$. This means that we are observing $m_{1,n}$ genes which appear once, $m_{2,n}$ genes which appear twice, etc.

Under the Bayesian nonparametric model \eqref{eq:bnpmodel}, the first issue to face is represented by the specification of the parameter $(\alpha,\theta)$ characterizing the prior $\Pi$. This is typically achieved by adopting an empirical Bayes procedure in order to obtain an estimate $(\hat\alpha,\hat\theta)$ of $(\alpha,\theta)$. Specifically we fix $(\alpha,\theta)$ so to maximize the likelihood function of the model \eqref{eq:bnpmodel} under the observed sample, namely
\begin{displaymath}
(\hat\alpha,\hat\theta)=\operatorname*{arg\,max}_{(\alpha,\theta)}\left\{\frac{\prod_{i=0}^{j-1}(\theta+i\alpha)}{(\theta)_{n}}\prod_{i=1}^j(1-\alpha)_{(n_{i}-1)}\right\}.
\end{displaymath}
Alternatively, one could specify a prior distribution for $(\alpha,\theta)$. Here we adopt a less elaborate specification of the parameter $(\alpha,\theta)$. We choice $\alpha=1/2$ and then we set $\theta$ such that $\E_{1/2,\theta}[K_{n}]=(2\theta)(((\theta+2^{-1})_{n}/(\theta)_{n})-1)=j$. Empirical investigations with simulated data suggests that $\alpha=1/2$ is always a good choice when no precise prior information is available. See Lijoi et al. \cite{Lij(07)} for details. This approach gives $(\alpha,\theta)=(1/2,206.75)$ for the Mastigamoeba non-normalized and  $(\alpha,\theta)=(1/2,132.92)$ for the Mastigamoeba normalized.

For the Mastigamoeba non-normalized and normalized cDNA libraries, Table 1 reports the exact estimate $\mathcal{D}_{m}^{(n)}$ and the corresponding large $m$ approximate estimates under the uncorrected normalization rate $m^{\alpha-1}$ and the corrected normalization rate \eqref{rate_corr1}. These are denoted by $\bar{\mathcal{D}}_{m}^{(n)}$ and $\tilde{\mathcal{D}}_{m}^{(n)}$, respectively. In a similar fashion, Table 2 reports the exact estimate $m^{-1}\mathcal{M}_{1,m}^{(n)}$ and the corresponding large $m$ approximate estimates under the uncorrected normalization rate $m^{\alpha}$ and the corrected normalization rate \eqref{rate_corr2}, respectively. These are denote by $m^{-1}\bar{\mathcal{M}}_{1,m}^{(n)}$ and $m^{-1}\tilde{\mathcal{M}}_{1,m}^{(n)}$, respectively. See \eqref{eq4_discuss} for details.

\medskip
\begin{center}
(Table 1 and Table 2 about here)
\end{center}
\medskip

Table 1 and Table 2 clearly show that the corrected normalization rates \eqref{rate_corr1} and \eqref{rate_corr2} are of fundamental importance when the additional sample size $m$ is not much larger than the sample size $n$ and the parameter $\theta$. Figure 1 and Figure 2 show the large deviation approximations  \eqref{eq5_discuss11} and \eqref{eq5_discuss12} of the estimate $\mathcal{D}_{m}^{(n)}(x)$.

\medskip
\begin{center}
(Figure 1 and Figure 2 about here)
\end{center}
\medskip


\clearpage

\begin{small}
\noindent Table 2. \textit{Exact estimate and corresponding asymptotic estimates under the uncorrected and corrected normalization rate}.\\[-0.2cm]
\begin{center}
\fbox{%
\tabcolsep=0.2cm
\footnotesize{
\begin{tabular}{*{5}{c}}
cDNA Library & $m$ & $\mathcal{D}_{m}^{(n)}$  &  $\bar{\mathcal{D}}_{m}^{(n)}$ & $\tilde{\mathcal{D}}_{m}^{(n)}$\\[0.1cm]\hline\\[-0.4cm]
Mastigamoeba non-normalized
&$\lfloor 100^{-1}n\rfloor$&0.472&5.438&0.472\\
&$\lfloor 10^{-1}n\rfloor$&0.456&1.696&0.456\\
&$n$&0.357&0.538&0.357\\
 $(n=715)$&10$n$&0.160&0.314&0.160\\
 &100$n$&0.054&0.054&0.054\\[0.4cm]
Mastigamoeba normalized
&$\lfloor 100^{-1}n\rfloor$&0.516&5.770&0.516\\
&$\lfloor 10^{-1}n\rfloor$&0.500&1.923&0.500\\
&$n$&0.397&0.606&0.397\\
$(n=363)$&10$n$&0.180&0.288&0.180\\
&100$n$&0.060&0.061&0.060\\[0.1cm]
\end{tabular}}}
\end{center}
\end{small}

\bigskip

\begin{small}
\noindent Table 2. \textit{Exact estimate and corresponding asymptotic estimates under the uncorrected and corrected normalization rate}.\\[-0.4cm]
\begin{center}
\fbox{%
\tabcolsep=0.2cm
\footnotesize{
\begin{tabular}{*{5}{c}}
cDNA Library & $m$ & $m^{-1}\mathcal{M}_{1,m}^{(n)}$   & $m^{-1}\bar{\mathcal{M}}_{1,m}^{(n)}$ & $m^{-1}\tilde{\mathcal{M}}_{1,m}^{(n)}$\\[0.1cm]\hline\\[-0.2cm]
Mastigamoeba non-normalized&$\lfloor 100^{-1}n\rfloor$&54.268&5.438&54.268\\
&$\lfloor 10^{-1}n\rfloor$&5.213&1.696&5.213\\
&$n$&0.752&0.538&0.752\\
 $(n=715)$&10$n$&0.178&0.314&0.178\\
 &100$n$&0.054&0.054&0.054\\[0.4cm]
Mastigamoeba normalized&$\lfloor 100^{-1}n\rfloor$&50.316&5.770&50.316\\
&$\lfloor 10^{-1}n\rfloor$&5.865&1.923&5.865\\
&$n$&0.812&0.606&0.812\\
$(n=363)$&10$n$&0.199&0.288&0.199\\
&100$n$&0.061&0.061&0.061\\[0.1cm]
\end{tabular}}}
\end{center}
\end{small}

\clearpage

\begin{small}
\noindent Figure 1. \textit{Mastigamoeba non-normalized. Large deviation approximations of the estimate $\mathcal{D}_{m}^{(715)}(x)$ under the uncorrected (blue line) and corrected (red line) normalization rate.}
\begin{figure}[ht!]
\centering
\includegraphics[width=1\linewidth]{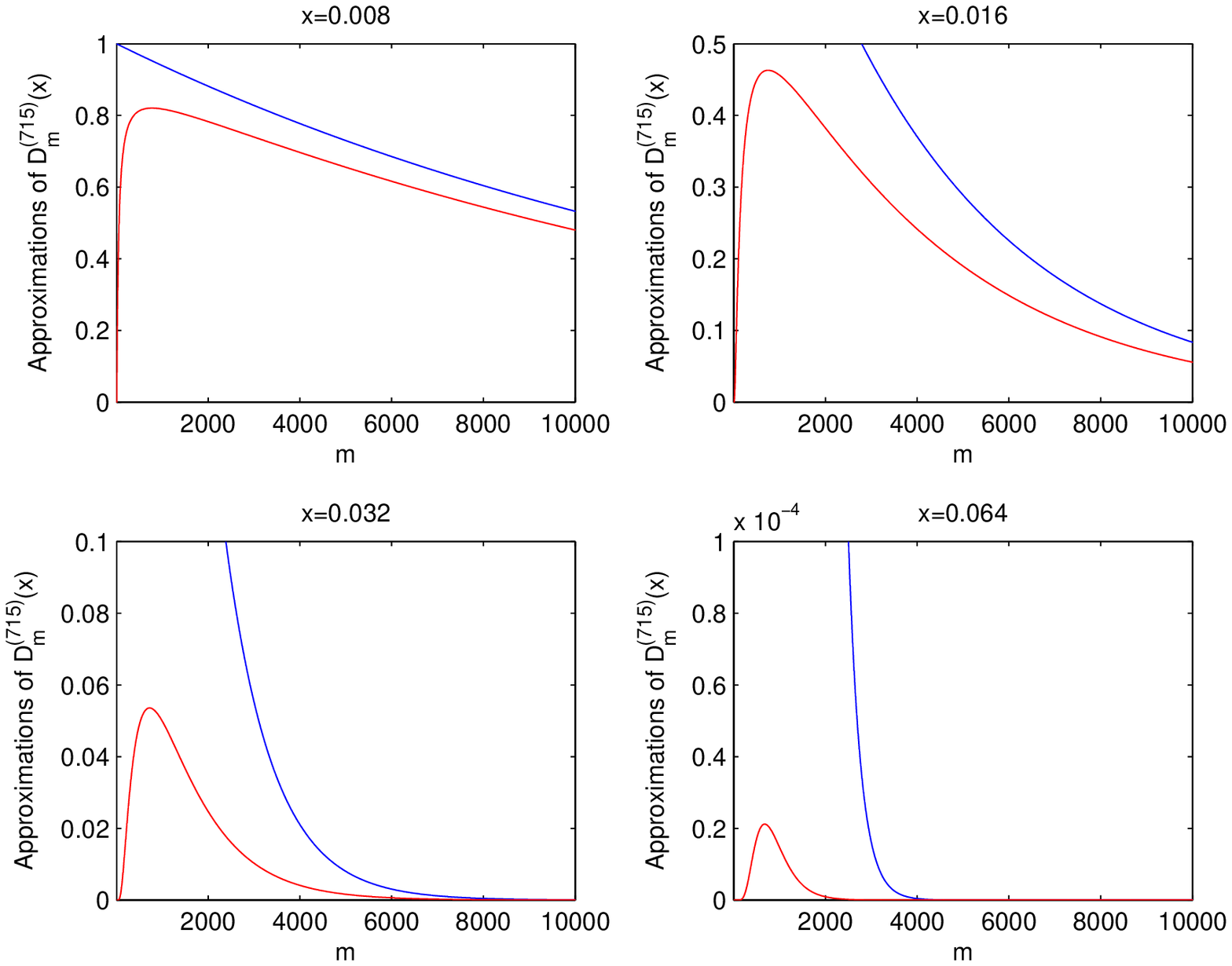}
\end{figure}
\end{small}

\clearpage

\begin{small}
\noindent Figure 2. \textit{Mastigamoeba normalized. Large deviation approximations of the estimate $\mathcal{D}_{m}^{(363)}(x)$ under the uncorrected (blue line) and corrected (red line) normalization rate.}
\begin{figure}[ht!]
\centering
\includegraphics[width=1\linewidth]{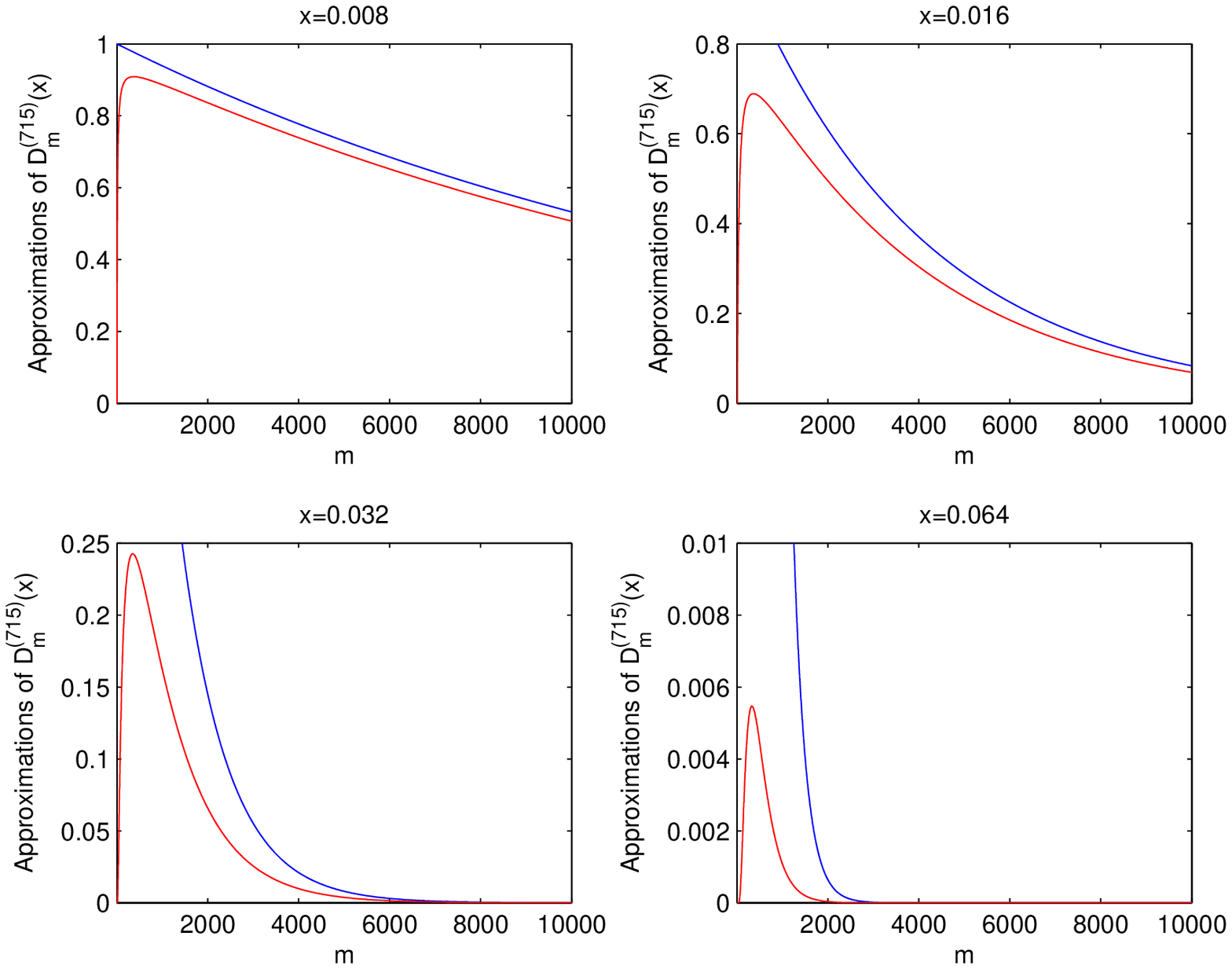}
\end{figure}
\end{small}


\begin{thebibliography}{9}

\bibitem{Arr(92)}
\textsc{Arratia, R., Barbour, A.D. and Tavar\'e, S.} (1992).  Poisson process approximations for the Ewens sampling formula.
\textit{Ann. Appl. Probab.} {\bf 2}, 519--535.

\bibitem{Arr(03)}
\textsc{Arratia, R., Barbour, A.D. and Tavar\'e, S.} (2003). \textit{Logarithmic combinatorial structures: a probabilistic approach}. EMS Monograph in Mathematics.

\bibitem{Bac(13)} 
\textsc{Bacallado, S., Favaro, S. and Trippa, L.} (2013). Looking-backward probabilities for Gibbs-type exchangeable random partitions. \textit{Bernoulli}, to appear.

\bibitem{Bar(09)}
\textsc{Barbour, A.D. and Gnedin, A.V.} (2009). Small counts in the infinite occupancy scheme. \textit{Electron. J. Probab.}, \textbf{13}, 365--384.

\bibitem{Cha(05)}
\textsc{Charalambides, C.A.} (2005). \textit{Combinatorial methods in discrete distributions}. Wiley Interscience.

\bibitem{DZ98}
\textsc{Dembo, A. and Zeitouni, O.} (1998) \textit{Large deviations techniques and applications}. Springer, New York.

\bibitem{Din(92)}
\textsc{Dinwoodie, I.H. and Zabell, S.L.} (1992). Large deviations for exchangeable random vectors. \textit{Ann. Probab.}, \textbf{20}, 1147-1166

\bibitem{Ewe(72)}
\textsc{Ewens, W.J.} (1972). The sampling theory of selectively neutral alleles. \textit{Theor. Popul. Biol.}, \textbf{3}, 87--112.

\bibitem{Fav(09)}
\textsc{Favaro, S., Lijoi, A., Mena, R.H. and Pr\"unster, I.} (2009). Bayesian nonparametric inference for species variety with a two parameter Poisson-Dirichlet process prior. \textit{J. Roy. Statist. Soc. Ser. B}, \textbf{71}, 993--1008.

\bibitem{Fav(13)}
\textsc{Favaro, S., Lijoi, A. and Pr\"unster, I.} (2013). Conditional formulae for Gibbs-type exchangeable random partitions. \textit{Ann. Appl. Probab.}, \textbf{23}, 1721--1754.

\bibitem{Fav(14)}
\textsc{Favaro, S., Feng, S.} (2014). Asymptotics for the conditional number of blocks in the Ewens-Pitman sampling model. \textit{Electron. J. Probab.}, \textbf{19}, 1--15

\bibitem{Feng(07)}
\textsc{Feng, S.} (2007). Large deviations associated with Poisson-Dirichlet distribution and Ewens sampling formula. \textit{Ann. Appl. Probab.}, \textbf{17},1570--1595.

\bibitem{Feng10}
\textsc{Feng, S.} (2010). \textit{The Poisson-Dirichlet distribution and related topics: models and asymptotic behaviors}, Springer,  Heidelberg.

\bibitem{Fen(98)}
\textsc{Feng, S. and Hoppe, F.M.} (1998). Large deviation principles for some random combinatorial structures in population genetics and Brownian motion. \textit{Ann. Appl. Probab.}, \textbf{8}, 975--994.

\bibitem{Goo(53)} 
\textsc{Good, I.J.} (1953). The population frequencies of species and the estimation of population parameters. \textit{Biometrika}, \textbf{40}, 237--264.

\bibitem{Goo(56)}  
\textsc{Good, I.J. and Toulmin, G.H.} (1956). The number of new species, and the increase in population coverage, when a sample is increased. \textit{Biometrika}, \textbf{43}, 45--63.

\bibitem{Gri(07)} 
\textsc{Griffiths, R.C. and Span\`o, D.} (2007). Record indices and age-ordered frequencies in exchangeable Gibbs partitions. \textit{Electron. J. Probab.}, \textbf{12}, 1101--1130.

\bibitem{Kor(73)} 
\textsc{Korwar, R.M. and Hollander, M.} (1973). Contribution to the theory of Dirichlet processes. \textit{Ann. Probab.}, \textbf{1},705--711.

\bibitem{Lij(07)}
\textsc{Lijoi, A., Mena, R.H. and Pr\"unster, I.} (2007). Bayesian nonparametric estimation of the probability of discovering a new species \textit{Biometrika}, \textbf{94}, 769--786.

\bibitem{Lij(08)} 
\textsc{Lijoi, A., Pr\"unster, I. and Walker, S.G.} (2008). Bayesian nonparametric estimators derived from conditional Gibbs structures. \textit{Ann. Appl. Probab.}, \textbf{18}, 1519--1547.

\bibitem{Per(92)}
\textsc{Perman, M., Pitman, J. and Yor, M.} (1992). Size-biased sampling of Poisson point processes and excursions. \textit{Probab. Theory Related Fields}, \textbf{92}, 21--39.

\bibitem{Pit(95)}
\textsc{Pitman, J.} (1995). Exchangeable and partially exchangeable random partitions. \textit{Probab. Theory Related Fields}, \textbf{102}, 145--158.

\bibitem{Pit(96a)}  
\textsc{Pitman, J.} (1996). Some developments of the Blackwell-MacQueen urn scheme. In {\it Statistics, Probability and Game Theory} (T.S. Ferguson, L.S. Shapley and J.B. MacQueen Eds.),
Hayward: Institute of Mathematical Statistics, 245--267.

\bibitem{Pit(96)}
\textsc{Pitman, J.} (1997). Partition structures derived from Brownian motion and stable subordinators. \textit{Bernoulli}, \textbf{3}, 79--66.

\bibitem{Pit(97)}
\textsc{Pitman, J. and Yor, M.} (1997). The two parameter Poisson-Dirichlet distribution derived from a stable subordinator. \textit{Ann. Probab.}, \textbf{25}, 855--900.

\bibitem{Pit(06)}
\textsc{Pitman, J.} (2006). \textit{Combinatorial stochastic processes.} Ecole d'Et\'e de Probabilit\'es de Saint-Flour XXXII. Lecture Notes in Mathematics N. 1875, Springer-Verlag, New York.

\bibitem{Sch(10)}
\textsc{Schweinsberg, J.} (2010).
The number of small blocks in exchangeable random partitions. \textit{ALEA Lat. Am. J. Probab. Math. Stat}. \textbf{7}, 217--242.

\bibitem{Sus(04)}
\textsc{Susko, E. and Roger, A.J.} (2004).
Estimating and comparing the rates of gene discovery and expressed sequence tag (EST) frequencies in EST surveys. \textit{Bioinformatics}, \textbf{20}, 2279--2287.

\end{thebibliography}
\end{document}